\documentclass[11pt]{article}%
\usepackage{float} 
\usepackage{fullpage}
\usepackage{amsfonts}
\usepackage{amsmath}
\usepackage{amssymb}
\usepackage{amsbsy}
\usepackage{amsthm}
\usepackage{mathtools}
\usepackage{epsfig}
\usepackage{graphicx}
\usepackage{colordvi}
\usepackage{graphics}
\usepackage{color}
\usepackage{bm}
\usepackage{verbatim} 
\numberwithin{equation}{section}

\newtheorem{theorem}{Theorem}[section]
\newtheorem{thm}{Theorem}[section]
\newtheorem{lem}{Lemma}[section]

\newtheorem{algorithm}[thm]{Algorithm}

\newtheorem{remark}[thm]{Remark}

\begin{document}

\title{An efficient splitting iteration for a CDA-accelerated solver for incompressible flow problems}

\author{
Victoria L. Fisher \thanks{\small School of Mathematical and Statistical Sciences, Clemson University, Clemson, SC, 29364. Partially supported by Department of Energy grant DMS DE-SC0025292 (vluongo@clemson.edu)}
\and
Leo G. Rebholz\thanks{\small School of Mathematical and Statistical Sciences, Clemson University, Clemson, SC, 29364.  Partially supported by Department of Energy grant DMS DE-SC0025292  (rebholz@clemson.edu)}
\and 
Duygu Vargun\thanks{\small Computer Science and Mathematics Division, Oak Ridge National Laboratory, Oak Ridge, TN 37831. Partially supported by Department of Energy grant DE-AC05-00OR22725 (vargund@ornl.gov)\\~\\
This research is based upon work supported by the U.S. Department of Energy (DOE) Office of Science, Office of Advanced Scientific Computing Research, as part of their Applied Mathematics Research Program. The work was performed at the Oak Ridge National Laboratory, which is managed by UT-Battelle, LLC under Contract No. DE-AC05-00OR22725. The U.S. Government retains and the publisher, by accepting the article for publication, acknowledges that the U.S. Government retains a non-exclusive, paid-up, irrevocable, world-wide license to publish or reproduce the published form of this manuscript, or allow others to do so, for the U.S. Government purposes. The DOE will provide public access to these results of federally sponsored research in accordance with the DOE Public Access Plan  (http://energy .gov /downloads /doe -public -access -plan).}
}
\maketitle

\begin{abstract}
We propose, analyze, and test an efficient splitting iteration for solving the incompressible, steady Navier-Stokes equations in the setting where partial solution data is known.   The (possibly noisy) solution data is incorporated into a Picard-type solver via continuous data assimilation (CDA).  Efficiency is gained over the usual Picard iteration through an algebraic splitting of Yosida-type that produces easier linear solves, and accuracy/consistency is shown to be maintained through the use of an incremental pressure and grad-div stabilization.    We prove that CDA scales the Lipschitz constant of the associated fixed point operator by $H^{1/2}$, where $H$ is the characteristic spacing of the known solution data.  This implies that CDA accelerates an already converging solver (and the more data, the more acceleration) and enables convergence of solvers in parameter regimes where the solver would fail (and the more data, the larger the parameter regime).  Numerical tests illustrate the theory on several benchmark test problems and show that the proposed efficient solver gives nearly identical results in terms of number of iterations to converge; in other words, the proposed solver gives an efficiency gain with no loss in convergence rate.
\end{abstract}

\section{Introduction}

Predicting fluid flow behavior is a ubiquitous task across many areas of science and engineering, and the typical approach is to numerically solve the governing equations for the particular flow.  The most common type, and the type we consider herein, is incompressible Newtonian flows, which are governed by the Navier-Stokes equations (NSE).  To this end, we consider herein nonlinear solvers for the steady NSE, which are given in a domain $\Omega \subset \mathbb{R}^d$ with $d =2, 3$ by
\begin{align}
   \label{eq1} u \cdot \nabla u + \nabla p - \nu \Delta u &= f,\\
   \label{eq2} \nabla \cdot u &= 0, \\
   \label{eq3} u |_{\partial \Omega} &= 0,
\end{align}
where $u$ represents the velocity, $p$ represents the pressure, $\nu$ is the kinematic viscosity, and $f$ is an external forcing.  While we consider the case of homogeneous Dirichlet boundary conditions in our analysis for simplicity, the analysis can be extended in a straightforward manner to nonhomogeneous mixed Dirichlet/Neumann boundary conditions.  Additionally, while we consider the case of nonlinear solvers for the steady NSE, our analysis is easily extendable to the nonlinear problems that arise at each time step of a time dependent NSE simulation.

Efficiently solving the NSE system \eqref{eq1}-\eqref{eq3} can be difficult, particularly when the Reynolds number $Re\sim \nu^{-1}$ gets larger.  A commonly used solver for the NSE is the Picard iteration, defined by
\begin{align}
   \label{pic1} u_{k-1} \cdot \nabla u_k + \nabla p_k - \nu \Delta u_k &= f,\\
   \label{pic2} \nabla \cdot u_k &= 0.
\end{align}
For problem parameters satisfying a particular smallness condition, which implies that the weak steady NSE are unique (more on this in Section 2), it is known that the Picard iteration is globally linear convergent \cite{GR86}.  This is in stark contrast to the Newton iteration which can converge quadratically but requires very good initial guesses to do so; as $Re$ increases, the initial guesses need to be even better.  Recent improvements to these solvers include Anderson acceleration applied to Picard \cite{PRX19,PR25} and nonlinear preconditioning of Newton with Picard or Anderson accelerated Picard \cite{PRTX25}, both of which can significantly improve convergence properties.  

We consider herein the setting where partial solution data is known, that is, the NSE velocity solution is known at $u(x_i)$ for $i=1,2,...,N$.  This solution data may come from physical measurements or perhaps from the holder of a very large scale HPC solution trying to pass the solution to another interested party, but the data is too large to send, so only partial data can be sent.  The partial solution data may also have noise or inaccuracies that can arise, for example, from measurement error or computational error.  One approach to incorporating the data into the nonlinear solvers proposed in \cite{LHRV23,GLNR24} is to use continuous data assimilation (CDA).   The CDA-Picard algorithm takes the form
	\begin{align*}
		-\nu \Delta u_{k}+u_{k-1}\cdot\nabla u_{k}+ \nabla p_{k} + \mu I_H(u_{k} - u)&={f}, \\
		\nabla\cdot {u}_{k}&=0.
	\end{align*}
Here $\mu>0$ is a user-defined nudging parameter, and $I_H$ is an appropriate interpolant with characteristic known partial solution data spacing of $H$ (more discussion of $I_H$ is given in Section 2).  We note that while the true solution $u$ remains unknown, $I_H(u)$ is known.   It is proven in \cite{LHRV23} that CDA scales the Lipschitz constant of the Picard fixed point operator by $H^{1/2}$, which shows that CDA accelerates convergence and enlarges the convergence basin with respect to $Re$.  For Newton, CDA is proven in \cite{LHRV23} to expand the convergence basin by $H^{-1/2}$, allowing Newton to be more robust with respect to $Re$ and initial guesses.   

The case of noisy data in CDA-Picard was studied in \cite{GLNR24}, and there it was found that the CDA-Picard convergence properties still held for the nonlinear residual.   This is essentially due to the noise being the same at each iteration and canceling out in the error equation, which thus resembles the case of no noise.  For the error, it was proven that convergence happens at the same rate as no noise but is stopped by a lower bound that scales with the size of the noise.

The purpose of this paper is to study a more efficient implementation of CDA-Picard \footnote{The same ideas we expect will also improve CDA-Newton, but this is nontrivial and will be considered in a separate work.}.  A major difficulty in solving Picard or CDA-Picard is the need at each iteration to solve a linear system of the form
\begin{equation} \label{system}
    \begin{pmatrix}
        A_k & B\\
        B^T & 0
    \end{pmatrix} 
    \begin{pmatrix}
        \hat{u_k}\\
        \hat{p_k}
    \end{pmatrix}  = 
    \begin{pmatrix}
        F\\
        0
    \end{pmatrix},
\end{equation}
where $\hat{u}_k$ and $\hat{p}_k$ represent the coefficient vectors of $u_k$ and $p_k$.  While there has been some progress to solving this type of system in recent years, e.g. \cite{benzi,elman:silvester:wathen,FMSW21}, it remains a difficult problem, and generally these linear system solves are by far the most computationally costly part of the simulation.  To alleviate this issue for the Picard iteration, a modification was proposed in \cite{RVX19} based on rewriting the Picard iteration in terms of an incremental pressure, using a Yosida-type algebraic splitting, and using grad-div stabilization to `absorb' the splitting error.  This iteration, called Incremental Picard Yosida (IPY), was analyzed and tested in \cite{RVX19}, and it was shown to work just as well as Picard in terms of convergence, but each iteration of IPY is much more efficient.  We note that while Picard {\it may} already use grad-div stabilization depending on the finite element choice (e.g., for Taylor-Hood it should be used \cite{JLMNR17}; for equal order it can be used, but the parameter choice is important \cite{GLRW12,HO25,JJLR13}), IPY {\it must} use grad-div stabilization.

The efficiency gain in IPY comes from its linear systems, which are given by 
\begin{equation} \label{system2}
    \begin{pmatrix}
        A_k & 0\\
        B^T & -B^T\tilde A^{-1}B
    \end{pmatrix} 
    \begin{pmatrix}
        I & A_k^{-1}B\\
        0 & I
    \end{pmatrix}
    \begin{pmatrix}
        \hat{u_k}\\
        \hat{\delta}^p_k
    \end{pmatrix}  = 
    \begin{pmatrix}
        F\\
        0
    \end{pmatrix},
\end{equation}
where $\hat\delta_k^p$ represents the coefficient vectors of $p_k - p_{k-1}$ and $\tilde A$ is the symmetric part of $A_k$.  Note that $\tilde A$ can be equivalently thought of as the matrix made from velocity contributions except for the (skew-symmetric) nonlinear term $u_{k-1}\cdot\nabla u_k$, and thus $\tilde A$ and $B^T \tilde A^{-1} B$ are symmetric positive definite (assuming $B$ is full column rank). Thus, this Stokes Schur complement can be easily solved with a conjugate gradient (CG) outer solver preconditioned with e.g. the pressure mass matrix \cite{benzi,HR13,BB12}, and CG can also be used as the inner solver.   Note that if $B$ is not full column rank, then the extra pressure degree of freedom in the system creates a single 0 eigenvalue in $B^T \tilde A^{-1} B$, but this does not cause issues in iterative linear solves \cite{benzi,elman:silvester:wathen}.   A critical feature of IPY is that $B^T \tilde A^{-1} B$ is exactly the same matrix at each iteration. This is in stark contrast to the Picard linear systems \eqref{system}, which have a Schur complement of  $B^T A_k^{-1} B$ that changes at each iteration.  Hence, IPY preconditioners for $\tilde A$ and $B^T \tilde A^{-1} B$ need to be constructed just once, while Picard preconditioners need to be constructed many times.

Thus, in order to gain efficiency over CDA-Picard, we consider the modified iteration CDA-IPY.  Although the high-level matrix structures of CDA-Picard and CDA-IPY are the same as Picard \eqref{system} and IPY \eqref{system2}, respectively, the differences from CDA are the nudging terms $\mu I_H u_{k+1}$ contributing to $A_k$ and $\tilde A$, and $\mu I_H u$ contributing to $F$.  While the efficiency gains of CDA-IPY over CDA-Picard are obvious, it is not all clear that the analysis of \cite{RVX19} that shows that IPY converges nearly as fast as Picard can be combined with the analysis that shows that CDA improves Picard's convergence rate by $H^{1/2}$.  Herein, we perform this analysis, which turns out to be rather technical but does yield the desired results: CDA is found to improve the convergence rate in IPY by a scaling of $H^{1/2}$, and it gives a similar convergence rate as CDA-Picard in \cite{LHRV23}.  However, our analysis reveals that the constants must be chosen carefully.  In particular, in CDA-Picard, one may choose $\mu$ to be large, but in CDA-IPY $\mu$ should be chosen $O(1)$.  We also analyze CDA-IPY in the case where the partial solution data is noisy; here we find that the CDA-IPY error and residual converge the same as when the data has no noise, but for the CDA-IPY error there is a lower bound on convergence that is on the order of the size of the noise in $L^2$.  Provided that the level of noise as a ratio of the solution is not large, we show that using the CDA-IPY ``solution'' with noisy data can be used to effectively seed Newton (without any CDA), even at large $Re$.

This paper is arranged as follows.  In Section 2, we give mathematical preliminaries and notation to make for a smoother analysis to follow.  Section 3 proves the convergence estimates for CDA-IPY in the cases of noise-free and noisy partial solution data.  Section 4 gives results of several numerical tests that illustrate the theory.  Finally, conclusions are drawn in Section 5.

\section{Mathematical Preliminaries}
\indent We consider a domain $\Omega \subset \mathbb{R}^d$, $d = 2, 3$, and denote the $L^2(\Omega)$ norm as $|| \cdot ||$ and inner product as $(\cdot, \cdot)$. All other norms will be labeled with subscripts.

Define the velocity and pressure spaces to be
\begin{align*}
    X & := H_0^1(\Omega) = \{ v\in H^1(\Omega),\ v|_{\partial\Omega}=0 \}, \\
    Q & := L^2_0(\Omega) = \{ q\in L^2(\Omega),\ \int_{\Omega} q=0 \}.
\end{align*}
We also define a divergence-free velocity space by
\begin{equation*}
    V := \{v \in X : (\nabla \cdot v, q) = 0, \forall q \in Q\}.
\end{equation*}
Recall that the Poincar\' e-Friedrichs inequality holds in $X$: there exists $C_P > 0$ that depends only on the size of $\Omega$ such that for every $v \in X$,
\begin{equation}
    \label{prelimeq7} ||v|| \leq C_P||\nabla v||.
\end{equation}
The dual space of $X$ will be denoted as $X'$, and the associated norm of $X'$ will be denoted as $||\cdot ||_{-1}$.  We use the notation $\langle \cdot, \cdot \rangle$ for the dual pairing of $X$ and $X'$.

\subsection{Finite Element Preliminaries}

\indent We consider a regular conforming triangulation $\tau_h$ of $\Omega$, where $h$ denotes the maximum element diameter. We denote $P_k$ as the space of degree $k \geq 1$ piecewise continuous polynomials on $\tau_h$. Likewise, we will denote $P_k^{disc}$ as the space of degree $k \geq 0$ piecewise discontinuous polynomials on $\tau_h$. 

We use discrete velocity-pressure spaces $(X_h, Q_h) \subset (X, Q)$ such that the LBB condition holds: there exists a constant $\beta$ that is independent of $h$ and satisfies
\begin{equation}
    \label{prelimeq6} \inf_{q \in Q_h} \sup_{v \in X_h} \frac{(\nabla \cdot v, q)}{||q||||\nabla v||} \geq \beta > 0.
\end{equation}
We also define the discretely divergence-free velocity space by 
\begin{equation*}
    V_h := \{v_h \in X_h : (\nabla \cdot v_h, q_h) = 0, \forall q_h \in Q_h\}.
\end{equation*}
Regarding the LBB condition, it is known, for example, that $(P_k,P_{k-1})$ $(k\ge 2)$ Taylor-Hood elements satisfy it \cite{BS08} as do $(P_k,P_{k-1}^{disc})$ Scott-Vogelius elements when $k=d$, and the mesh is constructed as a barycenter refinement of a triangular or tetrahedral mesh \cite{arnold:qin:scott:vogelius:2D,zhang:scott:vogelius:3D,JLMNR17}.  Our numerical tests will use these element choices; however, many other element choices are possible, see e.g. \cite{JLMNR17,BS08,BGHRR25}.

\subsection{NSE Preliminaries}

Define a skew-symmetric, trilinear operator $b: X \times X \times X \to \mathbb{R}$ as
\begin{equation*}
    b(\chi, v, w) = \frac{1}{2}(\chi \cdot \nabla v, w) - \frac{1}{2}(\chi \cdot \nabla w, v).
\end{equation*}
An important property of $b$ is that $b(\chi,v,v)=0$.  While we use here the classical skew-symmetric nonlinear form developed by a young Roger Teman \cite{T68}, the same results would hold for the rotational form \cite{GS98}, EMAC form \cite{CHOR17}, or simply the convective form if $\chi$ is restricted to $V$.

The following inequalities hold for $b$ \cite{Laytonbook}: there exists a constant $M$ that is only dependent on $\Omega$ such that
\begin{align}
    \label{prelimeq1} b(u, w, v) &\leq M ||u||^{\frac{1}{2}} ||\nabla u ||^{\frac{1}{2}} ||\nabla w|| ||\nabla v||,\\
    \label{prelimeq2} b(u, w, v) &\leq M ||\nabla u || ||\nabla w|| ||\nabla v||.
\end{align}

The NSE can now be written in finite element (FE) form as follows: find $(u, p) \in (X_h, Q_h)$ satisfying for all $(v, q) \in (X_h, Q_h)$, 
\begin{align}
   \label{eq4} \nu (\nabla u, \nabla v) + b(u, u, v) - (p, \nabla \cdot v)  + \gamma(\nabla \cdot u,\nabla \cdot v)&= \langle f, v \rangle,\\
   \label{eq5} (\nabla \cdot u, q) &= 0.
\end{align}
Note that the FE form of the NSE \eqref{eq4}-\eqref{eq5} contains a grad-div stabilization term with parameter $\gamma\ge 0$.  This term arises from adding $0=-\gamma\nabla (\nabla \cdot u)=0$ to the NSE. It is well known to penalize the lack of mass conservation and to reduce the effect of the pressure discretization error on the velocity error \cite{OR04,JJLR13,O02}.  For IPY, the use of grad-div stabilization was critical to remove the splitting error from the system \cite{RVX19}.

It is well known that the NSE \eqref{eq1}-\eqref{eq3} and FEM form of the NSE \eqref{eq4}-\eqref{eq5} are well-posed if the condition
\[
\alpha = M \nu^{-2} ||f||_{-1}<1.
\]
holds \cite{Laytonbook,temam}.  In fact, for any problem data, weak solutions to \eqref{eq1}-\eqref{eq3} exist, FEM solutions to \eqref{eq4}-\eqref{eq5} exist, and all such solutions satisfy
\begin{equation}
    \label{prelimeq3} ||\nabla u|| \leq \nu^{-1} ||f||_{-1}.
\end{equation}
The condition that $\alpha<1$ is a sufficient condition for uniqueness.  For sufficiently large $\alpha$, multiple solutions may exist \cite{Laytonbook}.

\begin{remark}
We do not make any assumptions on $\alpha$.  In the case of $\alpha$ large enough for multiple solutions to exist, we will prove convergence of CDA-IPY to the solution from which the partial solution data was taken.
\end{remark}

\subsection{CDA Preliminaries}

Denote by $\tau_H(\Omega)$ a coarse mesh of $\Omega$, and we require that nodes of $\tau_H$ are also nodes of $\tau_h$.
$I_H$ denotes an interpolant on $\Omega$ for which there exists a constant $C_1$ independent of $H$ such that 
\begin{align}
    \label{prelimeq4} ||I_Hv - v|| &\leq C_1H||\nabla v|| \quad \forall v \in X \\
    \label{prelimeq5} ||I_Hv|| &\leq C_1 ||v|| \quad \forall v \in X.
\end{align}
Examples of such $I_H$ are the $L^2$ projection onto $P_0(\tau_H)$ and the Scott-Zhang interpolant \cite{BS08}.  
The following lemma will be important for our analysis in the case of noise.
\begin{lem}\label{lem:geometric}
Suppose that the constants $r$ and $B$ satisfy $r > 1$ and $B \ge 0$. Then if the sequence of real numbers $\{a_m\}$ satisfies 
\begin{align}
r a_{m+1} \le a_m + B,
\end{align}
we have that
\begin{align*}
a_{m+1} \le a_0 \left(\frac{1}{r}\right)^{m+1} + \frac{B}{r-1}.
\end{align*}
\end{lem}

\begin{proof}
See \cite{LRZ19}.
\end{proof}

\subsection{IPY preliminaries}

We now briefly review the (CDA-)IPY procedure.  IPY was first proposed in \cite{RVX19} for the steady NSE and then studied with Anderson acceleration in \cite{LRX24}.  The derivation of CDA-IPY begins with the grad-div stabilized CDA-Picard iteration. This iteration is then reformulated to solve for the pressure update $\delta_k^p = p_k - p_{k-1}$: Find $(u_k,\delta_k^p)\in(X_h,Q_h)$ satisfying for all $(v,q)\in (X_h,Q_h)$,
\begin{align*}
   \nu (\nabla u_k, \nabla v) + \gamma (\nabla \cdot u_k, \nabla \cdot v)+  b(u_{k-1}, u_k, v)  \qquad & \\
    - (\delta_k^p, \nabla \cdot v)  + \mu (I_H u_k,I_Hv) &= \langle f, v \rangle + (p_{k-1},\nabla \cdot v) + \mu(I_H u,I_Hv),\\
    (\nabla \cdot u_k, q) &= 0.
\end{align*}
The difference between IPY and CDA-IPY is the term $ \mu (I_H u_k,I_Hv)$.  Note that a variational crime is potentially being made on the test function for the CDA terms if $I_H$ is not self-adjoint in the $L^2$ inner product.  However, using $I_H v$ is important to the analysis. Also, when adding stability or penalty terms to a system, it is more important how they affect the system and solution and less important how they came to be in a variational formulation.

This FE system produces the linear system
\begin{align*}
\begin{pmatrix}
A_k & B\\
B^T & 0 
\end{pmatrix} \begin{pmatrix}
\hat u_k \\ \hat \delta^p_k
\end{pmatrix} = \begin{pmatrix}
F_k  \\0
\end{pmatrix},
\end{align*}
where $A_k =G + S + D + N_k $ are the contributions  (after boundary condition implementation) from the grad-div matrix $G$, the stiffness matrix $S$, the CDA matrix $D$, and the matrix representing the nonlinear term contributions $N_k$, and   the matrix $B$ comes from the pressure term. The vector $F_k$ represents the contributions of the right hand side terms.  

We apply Yosida splitting via an inexact LU of the coefficient matrix as
\begin{align}
\begin{pmatrix}
A_k & B\\
B^T & 0 
\end{pmatrix} \approx\begin{pmatrix}
A_k & 0 \\ B^T & -B^T\tilde A^{-1}B
\end{pmatrix}\begin{pmatrix}
I & A_k^{-1}B \\ 0& I
\end{pmatrix},
\end{align}
where $\tilde A= G + S + D$ is SPD and constant at all iterations.  Replacing the coefficient matrix with this approximation produces the following 3-step linear solve process:
\begin{align*}
A_k \hat z_k &= F_k,\\
\begin{pmatrix}
 G + S + D  & B^T \\ B & 0
\end{pmatrix} \begin{pmatrix}
\hat \chi_k \\ \hat \delta_k^p
\end{pmatrix} & = \begin{pmatrix}
 0 \\ -B\hat z_k
\end{pmatrix},\\
\mbox{set } \hat p_k = \hat  \delta_k^p + \hat p_{k-1} \mbox{ and solve }
A_k \hat u_k &= F_k - B\hat p_{k}.
\end{align*}
This 3-step solve process can be cast back into the following FE problems.

\begin{enumerate}
\item[1)] Find $z_k\in X_h$ such that for all $v\in X_h$,
\begin{multline}
\label{bhy1}
\nu(\nabla z_k,\nabla v) + \gamma (\nabla \cdot z_k,\nabla \cdot v) + b(u_{k-1}, z_k,v) + \mu(I_H z_k,I_Hv) \\ = \langle f,v \rangle  + (p_{k-1},\nabla \cdot v) + \mu(I_H u,I_Hv).
\end{multline}
\item[2) ] Find $(\chi_k,\delta_k^p)\in (X_h,Q_h)$ such that for all $(v,q)\in (X_h,Q_h)$,
\begin{align}
\label{bhy2}
\nu(\nabla \chi_k,\nabla v) + \gamma (\nabla \cdot \chi_k,\nabla \cdot v) + \mu(I_H \chi_k,I_Hv)- (\delta_k^p,\nabla \cdot v)  & =0,\\
\label{bhy3}
(\nabla \cdot \chi_k, q) & = -(\nabla \cdot z_k, q).
\end{align}
\item[3)] Set $p_k = \delta_k^p + p_{k-1}$ and then find $u_k\in X_h$ such that for all $v\in X_h$,
\begin{multline}
\label{bhy4}
\nu(\nabla u_k,\nabla v) + \gamma (\nabla \cdot u_k,\nabla \cdot v) + \mu(I_H u_k,I_Hv)+ b(u_{k-1}, u_k, v)    = \\
\langle f,v \rangle  + \mu(I_H u,I_Hv)+(p_k,\nabla \cdot v).
\end{multline}
\end{enumerate}

Analysis of CDA-IPY is naturally more convenient in the FE form 3-step problems, while implementation is best understood with the 3-step linear algebraic system.  At each iteration, we must perform two solves with $A_k$ as the coefficient matrix, and we must perform one Stokes-type solve.  There are many effective solvers for Stokes-type problems \cite{elman:silvester:wathen}, and in our tests we solve it by solving the SPD Schur complement with CG outer and inner solvers and preconditioning the outer CG solve with the pressure mass matrix.

\section{Analysis of CDA-IPY}

We have provided sufficient background and notation to now formally define the CDA-IPY algorithm.  

\begin{algorithm}[CDA-IPY with no noise] \label{ipycda} Given $I_H(u)$, $f\in H^{-1}(\Omega)$, $\nu>0$, $\mu\ge 0$, and $u_0 \in X_h$ and $p_0 \in Q_h$,  for $k=1,2,3,...$: \\

\noindent \textit{Step k:} This consists of the following 3 steps:\\

\textit{k.1:} Find $z_k \in X_h$ satisfying for all $v \in X_h$,
\begin{equation} \label{k1}
    \nu (\nabla z_k, \nabla v) + \gamma (\nabla \cdot z_k, \nabla \cdot v) + b(u_{k-1}, z_k, v) + \mu(I_H z_k, I_H v) = \langle f, v \rangle + \mu(I_H u, I_H v) + (p_{k-1}, \nabla \cdot v).
\end{equation}
\indent \textit{k.2:} Find $(w_k, \delta_k^p) \in (X_h, Q_h)$ satisfying for all $(v, q) \in (X_h, Q_h)$,
\begin{align}
    \nu (\nabla w_k, \nabla v) + \gamma (\nabla \cdot w_k, \nabla \cdot v) -(\delta_k^p, \nabla \cdot v) + \mu(I_H w_k, I_H v) &= 0, \label{k2a} \\
    (\nabla \cdot w_k, q) &= -(\nabla \cdot z_k, q). \label{k2b}
\end{align}
\indent \textit{k.3:} Set $p_k := p_{k-1} + \delta_k^p$ and then find $u_k \in X_h$ satisfying for all $v \in X_h$,
\begin{equation}\label{k3}
    \nu (\nabla u_k, \nabla v) + \gamma (\nabla \cdot u_k, \nabla \cdot v) + b(u_{k-1}, u_k, v) + \mu(I_H u_k, I_H v) = \langle f, v \rangle + \mu(I_H u, I_H v) + (p_{k}, \nabla \cdot v).
\end{equation}
\end{algorithm}

\subsection{Convergence of CDA-IPY}
We now consider the convergence of Algorithm \ref{ipycda}, that is, CDA-IPY where there is no noise in the known partial solution data.  Solvability of each iteration is clear since Steps k.1 and k.3 follow immediately from Lax-Milgram, and Step k.2 follows from well known Stokes theory \cite{Laytonbook,GR79}.
Although we were unable to determine an a priori stability bound, such a bound is not necessary for our convergence results below, and thus solvability together with convergence implies stability.  For notational simplicity, we define a constant that arises in the analyses below by
\[
R := 204\sqrt{2}\nu^{-1}C_1\beta^{-2} + 24\sqrt{2} \nu^{-\frac{1}{2}}C_1\beta^{-2} + 204\sqrt{2}C_1\beta^{-2},
\]
where $\beta$ is the inf-sup constant and $C_1$ comes from the interpolation bound \eqref{prelimeq4}.  

\begin{theorem}[Convergence of CDA-IPY] \label{conv1}
 Let $(u,p)$ be a steady NSE solution, $I_H(u)$ be known and $(u_{k-1}, p_{k-1}) \in (X_h, Q_h)$ be given.  Assume that $u_{k-1}$ is sufficiently close to $u$ so that $\nu ||\nabla (u_{k-1} - u)||^2 \leq ||\nabla u||^2$, and assume that the problem parameters satisfy 
 \[
 H \le \min \left\{ \ \frac{1}{\alpha^2 R}, \frac{101}{204 \sqrt 2 C_1\alpha^2} \right\},
 \]
 $\max \{ \nu\alpha^2, \frac{\nu}{2C_1^2H^2} \} \le\mu$, and $\gamma \ge \max \{ 2\mu^2C_1^6C_P^2H^2\nu^{-1}, \nu \}$.
  
Then the step $k$ solution $(u_k, p_k)$ of Algorithm $3.1$  satisfies
\begin{equation*}
\begin{split}
     \frac{\nu}{4} ||\nabla \left(u - u_k\right)||^2 + \lambda||u - u_k||^2 + \frac{\gamma^{-1}}{2}||p - p_k||^2 \leq \alpha^2 HR \left( \frac{\nu}{4} ||\nabla \left(u - u_{k-1}\right)||^2 + \lambda|| u - u_{k-1}||^2 \right. \\ 
    \left. + \frac{\gamma^{-1}}{2}||p - p_{k-1}||^2 \right).
\end{split}
\end{equation*}
\end{theorem}

\begin{remark}
Theorem \ref{conv1} proves that if $H$ is sufficiently small (and other IPY parameters are appropriately chosen and $\| \nabla (u - u_0) \| \le \nu^{-1/2} \|\nabla u \|$), then CDA-IPY is a contractive algorithm.  Hence, if the initial velocity is chosen so that $\| \nabla (u - u_0) \| \le \nu^{-1/2} \|\nabla u \|$, then CDA-IPY converges at least linearly with rate $\alpha H^{1/2} R^{1/2}$.  The same scaling of the linear converge rate by $H^{1/2}$ was also observed in the CDA-Picard studies \cite{LHRV23,H25,GLNR24} and agrees with intuition that as more data is incorporated into the solver, the faster the convergence will be.
\end{remark}

\begin{remark}
The typical steady NSE uniqueness assumption $\alpha<1$ is not assumed in the proof, but instead we assume only that $H\alpha^2 < \frac{101}{204 \sqrt 2 C_1}$.  Thus, the CDA-IPY algorithm will work for large data and even in the case of non-unique NSE solutions: provided that $H$ is sufficiently small so that $H\alpha^2 < \frac{101}{204 \sqrt 2 C_1}$, CDA-IPY will converge to the steady NSE solution from which the partial solution data $I_H(u)$ was taken.
\end{remark}

\begin{remark}
The restriction that $\mu>\frac{\nu}{2C_1^2H^2}$ is common in CDA analyses \cite{AOT14,LRZ19}, but it is not believed to be sharp and is also not found to be a sharp bound in our tests.    From the assumption on $\gamma$ in the theorem, the grad-div parameter $\gamma$ appears to need to grow with $\mu$.  This suggests that one should not take $\mu$ too large, since having large $\gamma$ makes linear solves with $A_k$ and $\tilde A$ more difficult.  In our tests, we chose $\gamma=1$ together with both $\mu=1$ and $\mu=1000$. The method behaved very well and was not significantly affected by the larger choice of $\mu$.
\end{remark}

\begin{proof}
Let $z_k, w_k,\ \delta_k^p,\ u_k,\ p_k$ be step $k$ solutions from Algorithm \ref{ipycda}, and denote $e_k^z := u - z_k$ and $e_k^u := u - u_k$. We begin by subtracting \eqref{k1} from \eqref{eq4} to get the error equation for step $k.1$ to be
\begin{equation}
    \label{eq9} \nu (\nabla e_k^z, \nabla v) + \gamma (\nabla \cdot e_k^z, \nabla \cdot v) + b(e_{k-1}^u, u, v) + b(u_{k-1}, e_k^z, v) + \mu(I_H e_k^z, I_H v) = (p - p_{k-1}, \nabla \cdot v).
\end{equation}
Choosing $v = e_k^z$ vanishes the second nonlinear term and leaves
\begin{equation}
    \label{eq10} \nu ||\nabla e_k^z||^2 + \gamma ||\nabla \cdot e_k^z||^2 + \mu ||I_H e_k^z||^2 = (p - p_{k-1}, \nabla \cdot e_k^z) - b(e_{k-1}^u, u, e_k^z).
\end{equation}
Using Cauchy-Schwarz, \eqref{prelimeq1}, and \eqref{prelimeq3}, we can upper bound the right hand side terms to get 
\begin{equation}
    \label{eq11} \nu ||\nabla e_k^z||^2 + \gamma ||\nabla \cdot e_k^z||^2 + \mu ||I_H e_k^z||^2 \leq ||p - p_{k-1}|| ||\nabla \cdot e_k^z|| + \nu \alpha ||e_{k-1}^u||^{\frac{1}{2}} ||\nabla e_{k-1}^u||^{\frac{1}{2}}||\nabla e_k^z||.
\end{equation}
Using Young's inequality, we can further bound the right hand side terms and get
\begin{equation}
    \label{eq12} \nu ||\nabla e_k^z||^2 + \gamma ||\nabla \cdot e_k^z||^2 + \mu ||I_H e_k^z||^2 \leq \frac{\gamma^{-1}}{2}||p - p_{k-1}||^2 + \frac{\gamma}{2} ||\nabla \cdot e_k^z||^2 + \frac{\nu}{4}||\nabla e_k^z||^2 + \nu \alpha^2 ||e_{k-1}^u|| ||\nabla e_{k-1}^u||.
\end{equation}

We can also lower bound the left hand side terms of \eqref{eq12}, following ideas from \cite{GN20} by rewriting the left hand side and utilizing \eqref{prelimeq4} as well as the triangle inequality to obtain

\begin{align}
    \gamma ||\nabla \cdot e_k^z||^2 + \nu ||\nabla e_k^z||^2  + \mu ||I_H e_k^z||^2 &= \gamma ||\nabla \cdot e_k^z||^2 + \frac{\nu}{2} ||\nabla e_k^z||^2  + \mu ||I_H e_k^z||^2 + \frac{\nu}{2} ||\nabla e_k^z||^2 \nonumber \\
    &\geq \gamma ||\nabla \cdot e_k^z||^2 + \frac{\nu}{2C_1^2H^2} || e_k^z - I_H e_k^z||^2  + \mu ||I_H e_k^z||^2 + \frac{\nu}{2} ||\nabla e_k^z||^2 \nonumber \\
    &\geq \gamma ||\nabla \cdot e_k^z||^2 + \lambda\left( || e_k^z - I_H e_k^z||^2  + ||I_H e_k^z||^2\right) + \frac{\nu}{2} ||\nabla e_k^z||^2 \nonumber \\
   \label{eq13} & \geq \gamma ||\nabla \cdot e_k^z||^2 + \lambda|| e_k^z||^2 + \frac{\nu}{2} ||\nabla e_k^z||^2,
\end{align}
where $\lambda = \min \{\mu, \frac{\nu}{2C_1^2H^2} \} = \frac{\nu}{2C_1^2H^2}$ by the assumption on $\mu$. Combining \eqref{eq12} with \eqref{eq13} now gives 
\begin{equation}
    \label{eq14a} \frac{\gamma}{2} ||\nabla \cdot e_k^z||^2  + \frac{\nu}{4} ||\nabla e_k^z||^2 + \lambda|| e_k^z||^2 \leq \frac{\gamma^{-1}}{2}||p - p_{k-1}||^2 + \nu \alpha^2 ||e_{k-1}^u|| ||\nabla e_{k-1}^u||.
\end{equation}
Similarly, for Step $k.3$ of CDA-IPY \eqref{k3}, we obtain
\begin{equation}
    \label{eq14b} \frac{\gamma}{2} ||\nabla \cdot e_k^u||^2  + \frac{\nu}{4} ||\nabla e_k^u||^2 + \lambda|| e_k^u||^2 \leq \frac{\gamma^{-1}}{2}||p - p_{k}||^2 + \nu \alpha^2 ||e_{k-1}^u|| ||\nabla e_{k-1}^u||.
\end{equation}

Next, we work to obtain a bound on the term $||p - p_k||$ in \eqref{eq14b}, and this takes several steps.  We start by adding Steps $k.1$ \eqref{k1} and $k.2$ \eqref{k2a} to get 
\begin{multline*}
    \gamma (\nabla \cdot (w_k + z_k), \nabla \cdot v) + \nu (\nabla (w_k + z_k), \nabla v) \\ + \mu (I_H(w_k + z_k), I_Hv) = \langle f, v \rangle + \mu (I_Hu, I_Hv) + (p_k, \nabla \cdot v) - b(u_{k-1}, z_k, v).
\end{multline*}
Subtracting \eqref{eq4} from this, we obtain
\begin{align}
\begin{split}
     \label{eq16} \gamma (\nabla \cdot (w_k + z_k - u), \nabla \cdot v) + &\nu (\nabla (w_k + z_k - u), \nabla v) + \mu (I_H(w_k + z_k - u), I_Hv) \\
     &= (p_k - p, \nabla \cdot v) - b(e_{k-1}^u, e_k^z, v) - b(u, e_k^z, v) - b(e_{k-1}^u, u, v),
\end{split}
\end{align}
and then choosing $v = w_k + z_k - u$ provides us with
\begin{multline}
     \label{eq17} \gamma ||\nabla \cdot (w_k + z_k - u)||^2 + \nu ||\nabla (w_k + z_k - u)||^2 + \mu ||I_H(w_k + z_k - u)||^2\\
     = - b(e_{k-1}^u, e_k^z, w_k + z_k - u) - b(u, e_k^z,w_k + z_k - u) - b(e_{k-1}^u, u, w_k + z_k - u).
\end{multline}
Using \eqref{prelimeq1} on each of our right hand side terms along with Young's inequality and the assumption that $\nu ||\nabla e_{k-1}^u||^2 \leq ||\nabla u||^2$, we get that
\begin{multline}
     \label{eq18} \gamma ||\nabla \cdot (w_k + z_k - u)||^2 + \nu ||\nabla (w_k + z_k - u)||^2 + \mu ||I_H(w_k + z_k - u)||^2\\
     \leq (2\alpha^2 + 2\alpha^2\nu)||e_k^z||||\nabla e_k^z||+ 2\alpha^2\nu||e_{k-1}^u||||\nabla e_{k-1}^u|| + \frac{3\nu}{8}||\nabla (w_k + z_k - u)||^2.
\end{multline}
Moreover, utilizing the same techniques to obtain \eqref{eq13}, we can lower bound \eqref{eq18}, and thus we have
\begin{multline}
     \label{eq19b} \gamma ||\nabla \cdot (w_k + z_k - u)||^2 + \nu ||\nabla (w_k + z_k - u)||^2 + \lambda ||w_k + z_k - u||^2 \\
     \leq (16\alpha^2 + 16\alpha^2\nu)||e_k^z||||\nabla e_k^z||+ 16\alpha^2\nu||e_{k-1}^u||||\nabla e_{k-1}^u||.
\end{multline}

Isolating the pressure error term in \eqref{eq16} and applying Cauchy-Schwarz, \eqref{prelimeq1}, \eqref{prelimeq5}, and \eqref{prelimeq7} followed by the inf-sup condition bound \eqref{prelimeq6}, we get
\begin{multline}
    \label{eq21} \beta ||p - p_k|| \leq \frac{(p-p_k, \nabla \cdot v)}{||\nabla v||} \leq \gamma ||\nabla \cdot (w_k + z_k - u)|| + \nu || \nabla (w_k + z_k - u)|| + \mu C_1^2C_P||w_k + z_k - u|| \\ + M ||\nabla e_{k-1}^u|| ||e_k^z||^{\frac{1}{2}} ||\nabla e_k^z||^{\frac{1}{2}} + M ||\nabla u|| ||e_k^z||^{\frac{1}{2}} ||\nabla e_k^z||^{\frac{1}{2}} + M ||e_{k-1}^u||^{\frac{1}{2}} ||\nabla e_{k-1}^u||^{\frac{1}{2}} ||\nabla u||.
\end{multline}
Using the assumption that $\nu ||\nabla e_{k-1}^u||^2 \leq ||\nabla u||^2$ and \eqref{prelimeq3}, the bound \eqref{eq21} reduces to
\begin{multline}
   \label{eq23} \beta ||p - p_k|| \leq \gamma ||\nabla \cdot (w_k + z_k - u)|| + \nu || \nabla (w_k + z_k - u)|| + \mu C_1^2C_P||w_k + z_k - u||\\ + (\alpha \nu^{\frac{1}{2}} + \alpha \nu) ||e_k^z||^{\frac{1}{2}} ||\nabla e_k^z||^{\frac{1}{2}} + \alpha \nu||e_{k-1}^u||^{\frac{1}{2}} ||\nabla e_{k-1}^u||^{\frac{1}{2}}.
\end{multline}
Squaring both sides of \eqref{eq23} and applying Young's inequality gives us
\begin{multline*}
   \beta^2 ||p - p_k||^2 \leq 6 \left[\gamma^2 ||\nabla \cdot (w_k + z_k - u)||^2 + \nu^2 || \nabla (w_k + z_k - u)||^2 + \mu^2 C_1^4C_P^2||w_k + z_k - u||^2 \right. \\ \left. + (\alpha \nu^{\frac{1}{2}} + \alpha \nu)^2 ||e_k^z|| ||\nabla e_k^z|| + \alpha^2 \nu ^2||e_{k-1}^u||||\nabla e_{k-1}^u|| \right],
\end{multline*}
and using the assumption that $\gamma \geq \nu$, we obtain
\begin{multline}
   \label{eq25} \beta^2 ||p - p_k||^2 \leq 6\gamma \left[ \gamma ||\nabla \cdot (w_k + z_k - u)||^2 + \nu || \nabla (w_k + z_k - u)||^2 + \gamma^{-1} \mu^2 C_1^4C_P^2||w_k + z_k - u||^2\right.\\ \left. + (\alpha^2 + 2\alpha^2 \nu^{\frac{1}{2}} + \alpha^2 \nu) ||e_k^z|| ||\nabla e_k^z|| + \alpha^2 \nu ||e_{k-1}^u||||\nabla e_{k-1}^u|| \right].
\end{multline}
By our assumption that $\gamma \ge 2\mu^2C_1^6C_P^2H^2\nu^{-1}$, we can use the bound \eqref{eq19b} in \eqref{eq25} to find that
\begin{multline}
   \label{eq26} \beta^2 ||p - p_k||^2 \leq 6\gamma \left[(16\alpha^2+ 16\alpha^2\nu)||e_k^z||||\nabla e_k^z||+ 16\alpha^2\nu||e_{k-1}^u||||\nabla e_{k-1}^u|| \right. \\ \left. + (\alpha^2 + 2\alpha^2 \nu^{\frac{1}{2}} + \alpha^2 \nu) ||e_k^z|| ||\nabla e_k^z|| + \alpha^2 \nu ||e_{k-1}^u||||\nabla e_{k-1}^u|| \right].
\end{multline}
Combining like terms and dividing by $\beta^2$, we get the bound
\begin{equation}
   \label{eq27}||p - p_k||^2 \leq 6\beta^{-2}\gamma \left[(17\alpha^2 + 2\alpha^2 \nu^{\frac{1}{2}} + 17\alpha^2 \nu) ||e_k^z|| ||\nabla e_k^z|| + 17\alpha^2 \nu ||e_{k-1}^u||||\nabla e_{k-1}^u|| \right].
\end{equation}
Using \eqref{eq14a} to bound the $e_k^z$ terms, we now have
\begin{multline}
    \label{eq28} ||p - p_k||^2 \leq 6\beta^{-2}\gamma \left[\left(17\alpha^2 + 2\alpha^2 \nu^{\frac{1}{2}} + 17\alpha^2 \nu\right) \left(2\nu^{-\frac{1}{2}} \lambda^{-\frac{1}{2}} \right) \left(\frac{\gamma^{-1}}{2}||p - p_{k-1}||^2 + \nu \alpha^2 ||e_{k-1}^u|| ||\nabla e_{k-1}^u|| \right) \right. \\ \left. + 17\alpha^2 \nu ||e_{k-1}^u||||\nabla e_{k-1}^u|| \right].
\end{multline}
Multiplying both sides of \eqref{eq28} by $\frac{\gamma^{-1}}{2}$ and combining like terms gives 
\begin{multline}
    \label{eq29} \frac{\gamma^{-1}}{2}||p - p_k||^2 \leq 3\beta^{-2} \left[\left(34\alpha^2\nu^{-\frac{1}{2}}\lambda^{-\frac{1}{2}} + 4\alpha^2 \lambda^{-\frac{1}{2}} + 34\alpha^2 \nu^{\frac{1}{2}}\lambda^{-\frac{1}{2}}\right) \left(\frac{\gamma^{-1}}{2}||p - p_{k-1}||^2 \right) \right. \\ \left. + \left(34\alpha^4\nu^{\frac{1}{2}}\lambda^{-\frac{1}{2}} + 4\alpha^4\nu \lambda^{-\frac{1}{2}} + 34\alpha^4 \nu^{\frac{3}{2}}\lambda^{-\frac{1}{2}}+ 17\alpha^2\nu \right) ||e_{k-1}^u|| ||\nabla e_{k-1}^u|| \right].
\end{multline}
Now using \eqref{eq29} in \eqref{eq14b}, we obtain
\begin{multline}
    \label{eq30} \frac{\gamma}{2} ||\nabla \cdot e_k^u||^2  + \frac{\nu}{4} ||\nabla e_k^u||^2 + \lambda|| e_k^u||^2 \\ \leq 3\beta^{-2} \left[\left(34\alpha^2\nu^{-\frac{1}{2}}\lambda^{-\frac{1}{2}} + 4\alpha^2 \lambda^{-\frac{1}{2}} + 34\alpha^2 \nu^{\frac{1}{2}}\lambda^{-\frac{1}{2}}\right) \left(\frac{\gamma^{-1}}{2}||p - p_{k-1}||^2 \right) \right. \\ \left. + \left(34\alpha^4\nu^{\frac{1}{2}}\lambda^{-\frac{1}{2}} + 4\alpha^4\nu \lambda^{-\frac{1}{2}} + 34\alpha^4 \nu^{\frac{3}{2}}\lambda^{-\frac{1}{2}}+ 17\alpha^2\nu \right) ||e_{k-1}^u|| ||\nabla e_{k-1}^u|| \right] \\ + \nu \alpha^2 ||e_{k-1}^u|| ||\nabla e_{k-1}^u||.
\end{multline}
Adding this bound to \eqref{eq29} and reducing now gives
\begin{multline}
    \label{eq32} \frac{\gamma}{2} ||\nabla \cdot e_k^u||^2  + \frac{\nu}{4} ||\nabla e_k^u||^2 + \lambda|| e_k^u||^2 + \frac{\gamma^{-1}}{2}||p - p_k||^2 \\ \leq \left(204\alpha^2\nu^{-\frac{1}{2}}\lambda^{-\frac{1}{2}}\beta^{-2} + 24\alpha^2 \lambda^{-\frac{1}{2}}\beta^{-2} + 204\alpha^2 \nu^{\frac{1}{2}}\lambda^{-\frac{1}{2}}\beta^{-2}\right) \left(\frac{\gamma^{-1}}{2}||p - p_{k-1}||^2 \right) \\ + \left(204\alpha^4\nu^{\frac{1}{2}}\lambda^{-\frac{1}{2}}\beta^{-2} + 24\alpha^4\nu \lambda^{-\frac{1}{2}}\beta^{-2} + 204\alpha^4 \nu^{\frac{3}{2}}\lambda^{-\frac{1}{2}}\beta^{-2}+ 102\alpha^2\nu \beta^{-2} + \alpha^2\nu \right) ||e_{k-1}^u|| ||\nabla e_{k-1}^u||.
\end{multline}
Using the definition of $\lambda$ and since $\frac{\gamma}{2} ||\nabla \cdot e_k^u||^2 \ge 0$, we rewrite \eqref{eq32} as
\begin{multline}
    \label{eq33} \frac{\nu}{4} ||\nabla e_k^u||^2 + \lambda|| e_k^u||^2 + \frac{\gamma^{-1}}{2}||p - p_k||^2 \\ \leq \left(204\sqrt{2}\alpha^2\nu^{-1}C_1H\beta^{-2} + 24\sqrt{2}\alpha^2 \nu^{-\frac{1}{2}}C_1H\beta^{-2} + 204\sqrt{2}\alpha^2C_1H\beta^{-2}\right) \left(\frac{\gamma^{-1}}{2}||p - p_{k-1}||^2 \right) \\ + \left(204\sqrt{2}\alpha^4C_1H\beta^{-2} + 24\sqrt{2}\alpha^4\nu^{\frac{1}{2}}C_1H\beta^{-2} + 204\sqrt{2}\alpha^4 \nu C_1H\beta^{-2}+ 102\alpha^2\nu \beta^{-2} + \alpha^2\nu \right) ||e_{k-1}^u|| ||\nabla e_{k-1}^u||.
\end{multline}
Now using the definition of $R$, \eqref{eq33} reduces to
\begin{multline}
    \label{eq34} \frac{\nu}{4} ||\nabla e_k^u||^2 + \lambda|| e_k^u||^2 + \frac{\gamma^{-1}}{2}||p - p_k||^2\leq \alpha^2HR \left[ \frac{\gamma^{-1}}{2}||p - p_{k-1}||^2 \right. \\ \left. + \nu\left(\alpha^2+ 102 \beta^{-2}R^{-1}H^{-1} + R^{-1}H^{-1} \right) ||e_{k-1}^u|| ||\nabla e_{k-1}^u|| \right].
\end{multline}
Using Young's inequality on the last term, we have
\begin{multline}
    \label{eq35} \frac{\nu}{4} ||\nabla e_k^u||^2 + \lambda|| e_k^u||^2 + \frac{\gamma^{-1}}{2}||p - p_k||^2\leq \alpha^2 HR \left[ \frac{\gamma^{-1}}{2}||p - p_{k-1}||^2 \right. \\ \left. + \frac{\nu}{4}  ||\nabla e_{k-1}^u||^2 + \nu\left(\alpha^2+ 102 \beta^{-2}R^{-1}H^{-1} + R^{-1}H^{-1} \right)^2  ||e_{k-1}^u||^2 \right].
\end{multline}
It remains to show that 
\begin{equation}\label{needbound}
\nu\left(\alpha^2+ 102 \beta^{-2}R^{-1}H^{-1} + R^{-1}H^{-1} \right)^2 \leq \lambda = \frac{\nu}{2C_1^2H^2}.
\end{equation}   Since $R > 204\sqrt{2}C_1\beta^{-2}$, the desired inequality \eqref{needbound} is inferred if the bound
\[
H\alpha^2 + \frac{102\beta^{-2} + 1}{204\sqrt{2} C_1 \beta^{-2}}  \le \frac{1}{\sqrt 2C_1}
\]
holds.  We will now show that it indeed holds.

Since $\beta$ is an inf-sup constant and we assume homogeneous Dirichlet boundary conditions, it holds that $\beta\le 1$, and thus
\[
H\alpha^2 + \frac{102\beta^{-2} + 1}{204\sqrt{2} C_1 \beta^{-2}} = H\alpha^2 + \frac{102 + \beta^2}{204\sqrt{2} C_1}\le H\alpha^2 + \left(\frac{103}{204} \right) \frac{1}{\sqrt 2 C_1}.
\]
Thus, as the theorem assumes $H\alpha^2 \leq \frac{101}{204 \sqrt 2 C_1}$, the inequality \eqref{needbound} holds.  Thus, we obtain the estimate
%
%
%
%
\begin{equation}
    \label{eq38} \frac{\nu}{4} ||\nabla e_k^u||^2 + \lambda|| e_k^u||^2 + \frac{\gamma^{-1}}{2}||p - p_k||^2\leq \alpha^2 HR \left[ \frac{\nu}{4} ||\nabla e_{k-1}^u||^2 + \lambda|| e_{k-1}^u||^2 + \frac{\gamma^{-1}}{2}||p - p_{k-1}||^2 \right],
\end{equation}
which completes the proof.
\end{proof}

\subsection{CDA-IPY in the case of noisy partial solution data}

We now consider the case of partial solution data with noise.  That is, instead of knowing $I_H(u)$, we instead know $I_H(u+\epsilon)$, which comes from knowing $\left( u(x_i)+\epsilon(x_i) \right)$ at certain locations $x_i\in \Omega$.  The CDA-IPY algorithm can then be modified as follows.

\begin{algorithm}\label{CDA-IPY with noisy partial solution data}\label{alg2}

Given $u_0 \in X_h$ and $p_0 \in Q_h$, Step $k$ of the algorithm consists of the following 3 steps:\\

\textit{k.1:} Find $z_k \in X_h$ satisfying for all $v \in X_h$,
\begin{equation}
    \nu (\nabla z_k, \nabla v) + \gamma (\nabla \cdot z_k, \nabla \cdot v) + b(u_{k-1}, z_k, v) + \mu(I_H z_k, I_H v) = \langle f, v \rangle+ \mu(I_H (u+\epsilon), I_H v) + (p_{k-1}, \nabla \cdot v). \label{noise1}
\end{equation}
\indent \textit{k.2:} Find $(w_k, \delta_k^p) \in (X_h, Q_h)$ satisfying for all $(v, q) \in (X_h, Q_h)$,
\begin{align*}
    \nu (\nabla w_k, \nabla v) + \gamma (\nabla \cdot w_k, \nabla \cdot v) -(\delta_k^p, \nabla \cdot v) + \mu(I_H w_k, I_H v) &= 0,\\
    (\nabla \cdot w_k, q) &= -(\nabla \cdot z_k, q).
\end{align*}
\indent \textit{k.3:} Set $p_k := p_{k-1} + \delta_k^p$ and then find $u_k \in X_h$ satisfying for all $v \in X_h$,
\begin{equation}
    \nu (\nabla u_k, \nabla v) + \gamma (\nabla \cdot u_k, \nabla \cdot v) + b(u_{k-1}, u_k, v) + \mu(I_H u_k, I_H v) = \langle f, v \rangle+ \mu(I_H (u+\epsilon), I_H v) + (p_{k}, \nabla \cdot v). \label{noise3}
\end{equation}

\end{algorithm}

\begin{theorem}[Convergence of CDA-IPY with noisy partial solution data]\label{conv2}
Under the same assumptions as Theorem \ref{conv1}, Algorithm \ref{alg2} solutions satisfy
\begin{multline}
    \frac{\nu}{4} ||\nabla e_k^u||^2 + \lambda|| e_k^u||^2 + \frac{\gamma^{-1}}{2}||p - p_k||^2
    \leq   \left[ \frac{\nu}{4} ||\nabla e_{0}^u||^2 + \lambda|| e_{0}^u||^2 + \frac{\gamma^{-1}}{2}||p - p_{0}||^2 \right] \left(\alpha^2 H R\right)^{k+1} \\
    +  (4\alpha^2 H R + 1) \mu  \| I_H \epsilon \|^2.
    \label{thm2b}
\end{multline}
For $k$ sufficiently large and $\alpha^2 H R<1$, the CDA-IPY velocity solutions also satisfy
\begin{equation}
\| e_k^u \| \le \left( \frac{6 C_1^2 H^2 \mu}{\nu} \right)^{1/2} \| I_H \epsilon \|. \label{thm2c}
\end{equation}
\end{theorem}

\begin{remark}
Theorem \ref{conv2} shows that with noisy partial solution data, provided $\alpha^2HR<1$, the CDA-IPY will converge linearly to the level of the noise.  The estimate \eqref{thm2b} suggests that the convergence is limited by the size of the noise scaled by $\mu^{1/2}$. However, the $L^2$ velocity estimate \eqref{thm2c} shows a reduced scaling by $O\left( H \mu^{1/2} \nu^{-1/2} \right)$.

\end{remark}

\begin{proof}
We proceed similar to the case of no noise in the previous section by following the proof of Theorem \ref{conv1} but handling the noise terms as they arise.  We again set $e_k^z := u - z_k$ and $e_k^u := u - u_k$. Subtracting Step k.1 \eqref{noise1} from the NSE \eqref{eq4}, we get that
\begin{multline*}
    \nu (\nabla e_k^z, \nabla v) + \gamma (\nabla \cdot e_k^z, \nabla \cdot v) + b(e_{k-1}^u, u, v) + b(u_{k-1}, e_k^z, v) + \mu(I_H e_k^z, I_H v)  \\ = (p - p_{k-1}, \nabla \cdot v) - \mu(I_H \epsilon,I_H v). 
\end{multline*}
Choosing $v = e_k^z$ now provides
\begin{equation}
     \nu ||\nabla e_k^z||^2 + \gamma ||\nabla \cdot e_k^z||^2 + \mu ||I_H e_k^z||^2 = (p - p_{k-1}, \nabla \cdot e_k^z) - b(e_{k-1}^u, u, e_k^z)- \mu(I_H \epsilon,I_H e_k^z), \label{N2}
\end{equation}
and after majorizing the first two right hand side terms of \eqref{N2} as in the proof of Theorem \ref{conv1} and applying Young's inequality to the last term, we get the bound
\begin{equation}
    \nu ||\nabla e_k^z||^2 + \gamma ||\nabla \cdot e_k^z||^2 + \frac{\mu}{2}  ||I_H e_k^z||^2 \leq ||p - p_{k-1}|| ||\nabla \cdot e_k^z|| + \nu \alpha ||e_{k-1}^u||^{\frac{1}{2}} ||\nabla e_{k-1}^u||^{\frac{1}{2}}||\nabla e_k^z|| + \frac{\mu}{2} \| I_H \epsilon \|^2.  \label{N3}
\end{equation}
The left hand side of \eqref{N3} can be lower bounded in a manner analogous to what is done in \eqref{eq13} (the only difference is here we have $\frac{\mu}{2}$ instead of $\mu$) to obtain
\begin{equation}
  \nu ||\nabla e_k^z||^2 + \gamma ||\nabla \cdot e_k^z||^2 + \frac{\mu}{2}  ||I_H e_k^z||^2
  \ge 
  \gamma ||\nabla \cdot e_k^z||^2 + \lambda|| e_k^z||^2 + \frac{\nu}{2} ||\nabla e_k^z||^2, \label{N4}
\end{equation}
where $\lambda = \min \{\frac{\mu}{2}, \frac{\nu}{2C_1^2H^2} \}$.  Putting \eqref{N3} together with \eqref{N4} after bounding the first two right hand side terms of \eqref{N3} using Young's inequality gives us
\begin{equation}
    \frac{\gamma}{2} ||\nabla \cdot e_k^z||^2  + \frac{\nu}{4} ||\nabla e_k^z||^2 + \lambda|| e_k^z||^2 \leq \frac{\gamma^{-1}}{2}||p - p_{k-1}||^2 + \nu \alpha^2 ||e_{k-1}^u|| ||\nabla e_{k-1}^u|| + \frac{\mu}{2} \| I_H \epsilon \|^2. \label{N5}
\end{equation}

Notice that Step k.3 only differs from Step k.1 in the variable names; otherwise, the structure is the same.  Hence, proceeding analogously to the analysis above that produced \eqref{N5}, we obtain
\begin{equation}
    \frac{\gamma}{2} ||\nabla \cdot e_k^u||^2  + \frac{\nu}{4} ||\nabla e_k^u||^2 + \lambda|| e_k^u||^2 \leq \frac{\gamma^{-1}}{2}||p - p_{k}||^2 + \nu \alpha^2 ||e_{k-1}^u|| ||\nabla e_{k-1}^u|| + \frac{\mu}{2} \| I_H \epsilon \|^2. \label{N6}
\end{equation}

For Step k.2, there is no noise term.  Hence we utilize the Step k.2 analysis from the proof of Theorem \ref{conv1} to get the bound
\begin{equation}
   ||p - p_k||^2 \leq 6\beta^{-2}\gamma \left[(17\alpha^2 + 2\alpha^2 \nu^{\frac{1}{2}} + 17\alpha^2 \nu) ||e_k^z|| ||\nabla e_k^z|| + 17\alpha^2 \nu ||e_{k-1}^u||||\nabla e_{k-1}^u|| \right]. \label{N7}
\end{equation}
Using the bound \eqref{N5} in \eqref{N7}, multiplying both sides by $\frac{\gamma^{-1}}{2}$, and reducing yields
\begin{multline}
  \frac{\gamma^{-1}} {2}  ||p - p_k||^2 \leq 
   3\beta^{-2} \bigg( 
   (34\alpha^2\nu^{-1/2} \lambda^{-1/2} + 4\alpha^2 \lambda^{-1/2}  + 34\alpha^2 \nu^{1/2} \lambda^{-1/2} )(  \frac{\gamma^{-1}}{2}||p - p_{k-1}||^2 ) \\
   +   (34\alpha^4\nu^{1/2}\lambda^{-1/2} + 4\alpha^4 \nu\lambda^{-1/2} + 34\alpha^4 \nu^{3/2} \lambda^{-1/2} + 17 \alpha^2 \nu)   ||e_{k-1}^u|| ||\nabla e_{k-1}^u||   \\
      + (17\alpha^2\nu^{-1/2}\lambda^{-1/2} + 2\alpha^2\lambda^{-1/2} + 17\alpha^2 \nu^{1/2} \lambda^{-1/2}) \mu  \| I_H \epsilon \|^2   
   \bigg). \label{N8}
\end{multline}
Next, using \eqref{N8} in \eqref{N6} provides
\begin{multline}
     \frac{\gamma}{2} ||\nabla \cdot e_k^u||^2  + \frac{\nu}{4} ||\nabla e_k^u||^2 + \lambda|| e_k^u||^2  
     \leq 
     \\ 
 + 3 \beta^{-2} \bigg( 
   (34\alpha^2\nu^{-1/2} \lambda^{-1/2} + 4\alpha^2 \lambda^{-1/2}  + 34\alpha^2 \nu^{1/2} \lambda^{-1/2} )(  \frac{\gamma^{-1}}{2}||p - p_{k-1}||^2 ) \\
   +   (34\alpha^4\nu^{1/2}\lambda^{-1/2} + 4\alpha^4 \nu\lambda^{-1/2} + 34\alpha^4 \nu^{3/2} \lambda^{-1/2} + 17 \alpha^2 \nu + \frac{\nu\alpha^2\beta^2}{3})   ||e_{k-1}^u|| ||\nabla e_{k-1}^u||   \\
     + (17\alpha^2\nu^{-1/2}\lambda^{-1/2} + 2\alpha^2\lambda^{-1/2} + 17\alpha^2 \nu^{1/2} \lambda^{-1/2} + \frac{\beta^{2}}{6}) \mu  \| I_H \epsilon \|^2 
   \bigg).
  \label{N9}
\end{multline}
Adding \eqref{N9} to \eqref{N8} now gives us
\begin{multline}
     \frac{\nu}{4} ||\nabla e_k^u||^2 + \lambda|| e_k^u||^2  +  \frac{\gamma^{-1}} {2}  ||p - p_k||^2
     \leq 
     \\ 
 + 6\beta^{-2} 
   (34\alpha^2\nu^{-1/2} \lambda^{-1/2} + 4\alpha^2 \lambda^{-1/2}  + 34\alpha^2 \nu^{1/2} \lambda^{-1/2} )(  \frac{\gamma^{-1}}{2}||p - p_{k-1}||^2 ) \\
   +    6\beta^{-2} (34\alpha^4\nu^{1/2}\lambda^{-1/2} + 4\alpha^4 \nu\lambda^{-1/2} + 34\alpha^4 \nu^{3/2} \lambda^{-1/2} + 17 \alpha^2 \nu + \frac{\nu\alpha^2\beta^2}{6})   ||e_{k-1}^u|| ||\nabla e_{k-1}^u||   \\
     +  6\beta^{-2} (17\alpha^2\nu^{-1/2}\lambda^{-1/2} + 2\alpha^2\lambda^{-1/2} + 17\alpha^2 \nu^{1/2} \lambda^{-1/2} + \frac{\beta^{2}}{12}) \mu  \| I_H \epsilon \|^2,
  \label{N10}
\end{multline}
noting that we dropped a non-negative term on the left hand side and distributed the $\beta^{-2}$.  Note that since $\lambda = \frac{\nu}{2C_1^2 H^2}$ by assumption on $\mu$, \eqref{N10} is the same bound as \eqref{eq33} except for the noise term 
\[
 6\beta^{-2} (17\alpha^2\nu^{-1/2}\lambda^{-1/2} + 2\alpha^2\lambda^{-1/2} + 17\alpha^2 \nu^{1/2} \lambda^{-1/2} + \frac{\beta^{2}}{12}) \mu  \| I_H \epsilon \|^2,
\]
and thus using identical analysis from \eqref{eq33} to the end of the proof of Theorem \ref{conv1} (but carrying along the noise term) yields
\begin{multline}
    \frac{\nu}{4} ||\nabla e_k^u||^2 + \lambda|| e_k^u||^2 + \frac{\gamma^{-1}}{2}||p - p_k||^2\leq  \alpha^2 HR \left[ \frac{\nu}{4} ||\nabla e_{k-1}^u||^2 + \lambda|| e_{k-1}^u||^2 + \frac{\gamma^{-1}}{2}||p - p_{k-1}||^2 \right] \\
    +  (2\alpha^2 H R + \frac{1}{2} ) \mu  \| I_H \epsilon \|^2.
    \label{N11}
\end{multline}
Applying Lemma \ref{lem:geometric} to \eqref{N11} with $r=\frac{1}{\alpha^2 H R}$ and $B= (2\alpha^2 H R + \frac{1}{2} ) \mu  \| I_H \epsilon \|^2$ and using the assumption $\alpha^2 H R < \frac12$, we get that
\begin{multline}
    \frac{\nu}{4} ||\nabla e_k^u||^2 + \lambda|| e_k^u||^2 + \frac{\gamma^{-1}}{2}||p - p_k||^2
    \leq   \left[ \frac{\nu}{4} ||\nabla e_{0}^u||^2 + \lambda|| e_{0}^u||^2 + \frac{\gamma^{-1}}{2}||p - p_{0}||^2 \right] \left(\alpha^2 H R\right)^{k+1} \\
    +  (4\alpha^2 H R + 1) \mu  \| I_H \epsilon \|^2.
    \label{N12}
\end{multline}
This proves the first estimate of the theorem.  For the second, taking $k$ sufficiently large, using that $\mu>\frac{\nu}{2C_1^2H^2}$ and the Poincar\'e inequality gives the $L^2$ error estimate
\[
\left(C_P^{-2} + \frac{\nu}{2C_1^2 H^2} \right) \| e_k^u \|^2 \le  (4\alpha^2 H R + 1) \mu  \| I_H \epsilon \|^2 \le 3 \mu  \| I_H \epsilon \|^2.
\]
This bound reduces to
\[
\| e_k^u \| \le \left( \frac{6 C_1^2 H^2 \mu}{\nu} \right)^{1/2} \| I_H \epsilon \|,
\]
which finishes the proof.

\end{proof}

\section{Numerical Experiments}

We now provide results from several numerical tests that illustrate the theory above, compare to known methods, and investigate the overall usefulness of CDA-IPY.  First, we describe the test problems. Then, we give results for partial solution data with no noise. Finally, we present results with noise.  For the nudging in all tests, we use algebraic nudging from \cite{RZ21}, which is a particular implementation of the $L^2$ projection onto a coarse piecewise-constant finite element space.  All tests will use $u_0=0$ as the initial iterate, and unless otherwise mentioned, we use $\mu=1$ and $\gamma=1$.

\subsection{Description of test problems}

We now describe the test problems used in the numerical experiments to follow.  

\subsubsection{2D driven cavity}
The setup for the 2D driven cavity has the domain as the unit square $\Omega=(0,1)^2$, no forcing $f=0$, homogeneous Dirichlet boundary conditions for velocity enforced on the sides and bottom, and $[ 1,0 ]^T$ on the top.  Our tests use $Re:=\nu^{-1}$ with $Re$=5000 and 10000.  Solutions for these $Re$ from our tests below are shown in Figure \ref{dc2d} and are in agreement with solutions in the literature \cite{ECG05}.  

\begin{figure}[ht]
\center
$Re=5,000$ \hspace{1in}  $Re=10,000$ \\
\includegraphics[width = .3\textwidth, height=.28\textwidth,viewport=115 45 465 390, clip]{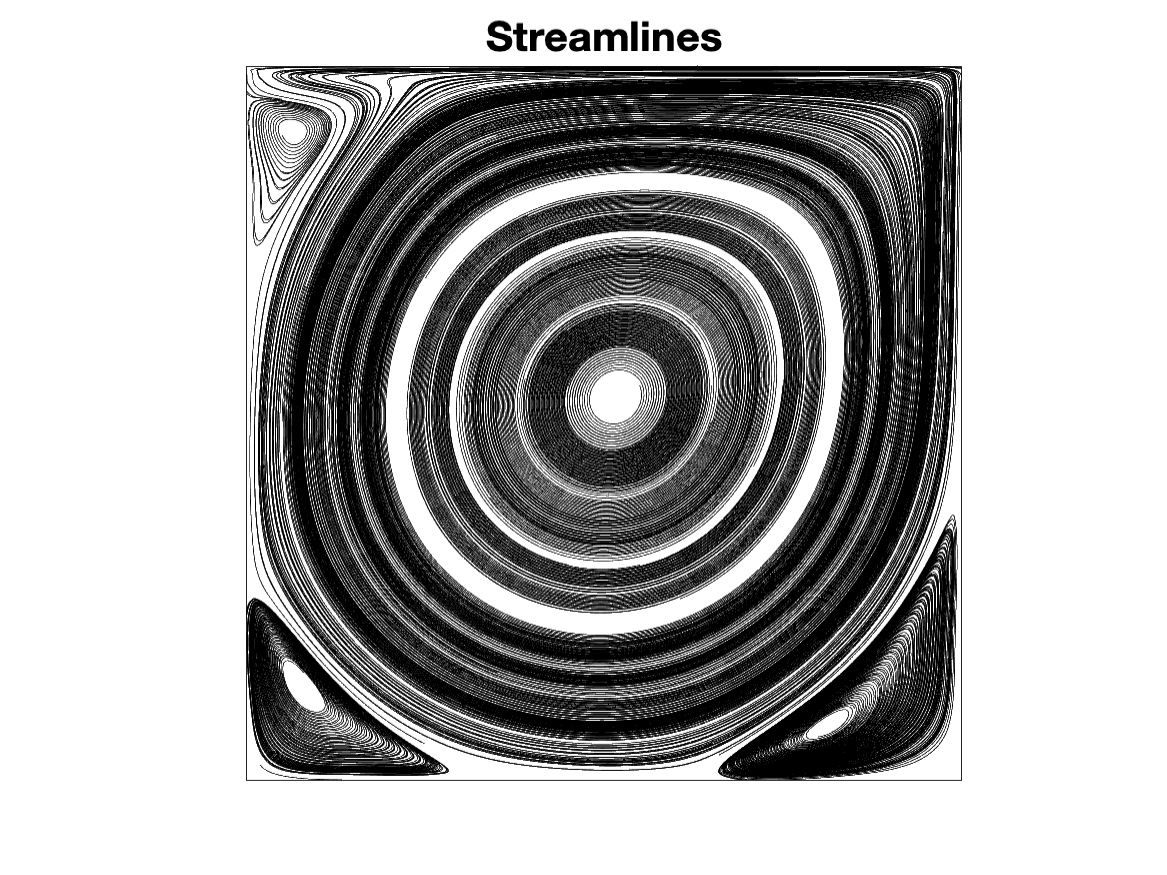}  
\includegraphics[width = .3\textwidth, height=.28\textwidth,viewport=115 45 465 390, clip]{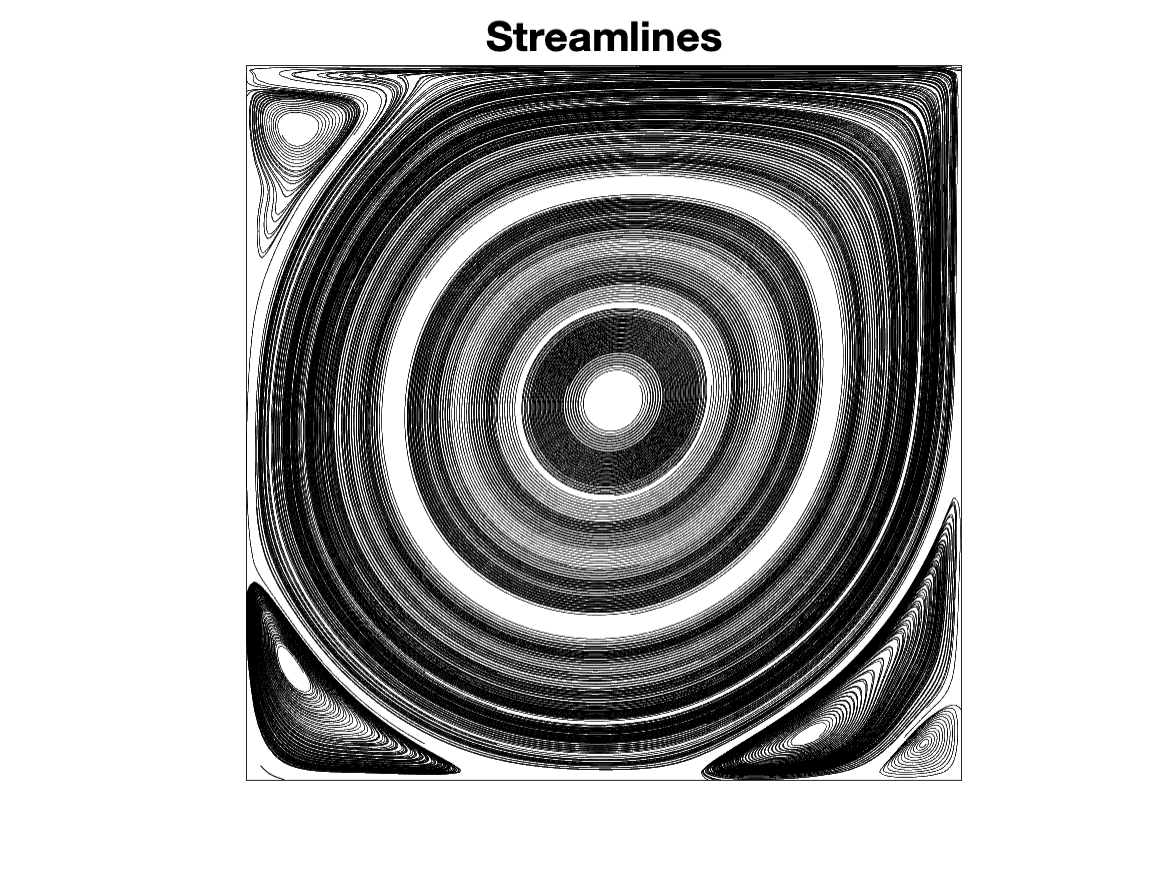}  
\caption{\label{dc2d} The plots above show streamlines of the solution of the 2D driven cavity problem at $Re=$5,000 (left) and 10,000 (right).}
\end{figure}

Our computations use $(P_2,P_1^{disc})$ Scott-Vogelius elements on barycenter refinements of $h=\frac{1}{128}$ and $\frac{1}{196}$ uniform triangular meshes, i.e. we use $\frac{1}{256}$ and $\frac{1}{392}$ which provide for 680K degrees of freedom (dof) and 1.6M dof, respectively.  We note that Scott-Vogelius elements of this order are known to be inf-sup stable on meshes with this structure \cite{arnold:qin:scott:vogelius:2D,JLMNR17}.  

\subsubsection{3D driven cavity}

The 3D lid-driven cavity benchmark problem is an analogue of the 2D problem described above: the domain is a unit cube, there is no forcing $f=0$, homogeneous Dirichlet boundary conditions for velocity are enforced on the walls, but at the top of the box, $u=[1,0,0]^T$ is enforced.  Our tests use $Re$=200, 1000, and 1500.

\begin{figure}[H]
			\centering
			\includegraphics[width = 0.95\textwidth, height=.27\textwidth,viewport=100 0 1100 300, clip]{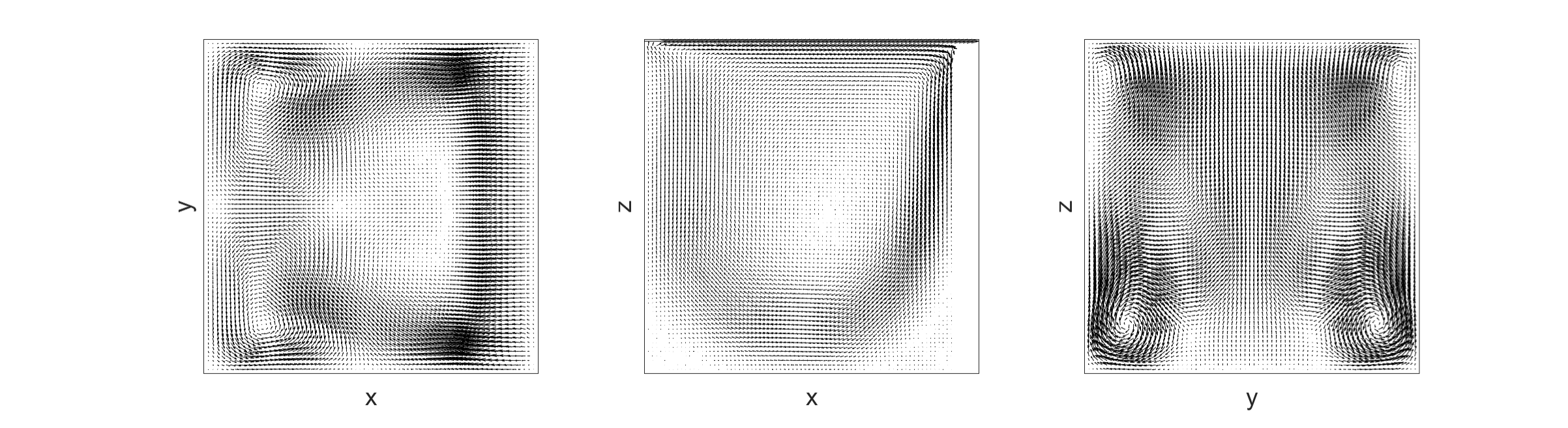}
			\caption{Shown above are the midsliceplane plots of solutions for the 3D driven cavity simulations at $Re$=1000.\label{fig:midsliceplanes}}
\end{figure}
		
To construct the mesh, we start with Chebychev points on [0,1] to create a $\mathcal{M}^3$ grid of rectangular boxes (we use $\mathcal{M}$=11 for $Re$=200 and 1000 and  $\mathcal{M}$=13 for $Re$=1500).  Each box is first split into 6 tetrahedra, and then each of these tetrahedra is split into 4 tetrahedra by a barycenter refinement / Alfeld split.  We use $(P_3, P_2^{disc})$ Scott-Vogelius elements, which provides for approximately 796K total dof for the $Re$=200  and 1000 tests and 1.3 million total dof for the $Re$=1500 tests.  This velocity-pressure pair is known to be LBB stable on such a mesh construction from \cite{zhang:scott:vogelius:3D}.   Solution plots found with this discretization matched those from the literature \cite{WongBaker2002}, and we show midspliceplanes of the $Re$=1000 solution in Figure \ref{fig:midsliceplanes}.

\subsubsection{Channel flow past a cylinder}

Our third test problem setup is the channel flow past a block benchmark. This problem has been widely studied both experimentally and numerically \cite{TGO15, R97, SDN99,HRV24}.  The computational domain is a rectangular channel of size $2.2\times0.41$ with a square obstacle of side length $0.1$ positioned at $(0.2, 0.2)$ relative to the bottom-left corner of the rectangle. We enforce no-slip velocity at the walls and block, and at the inflow and outflow, we enforce
\begin{align}
	u_1(0,y)=&u_1(2.2,y)=\frac{6}{0.41^2}y(0.41-y),\nonumber\\
	u_2(0,y)=&u_2(2.2,y)=0.
\end{align}

\begin{figure}[H]\label{cylpic}
  \centering
       \includegraphics[scale=0.11]{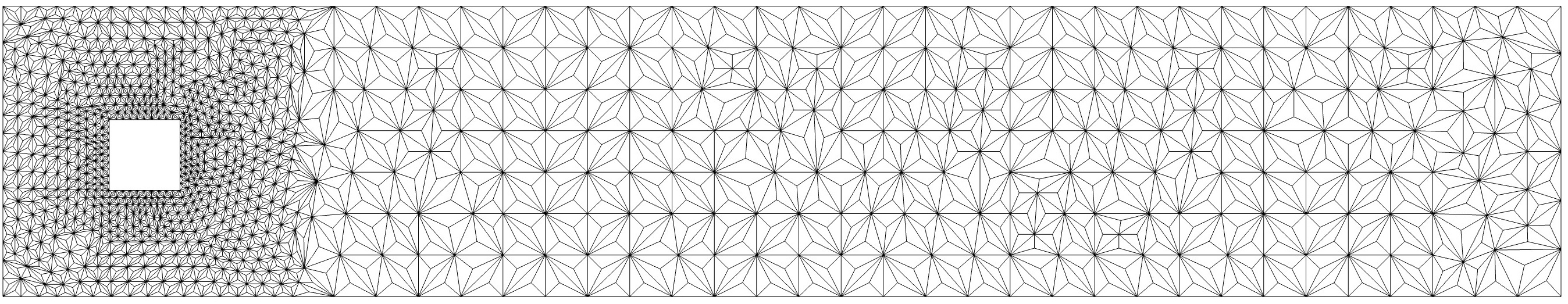}
  \caption{\label{cylpic} Shown above is a sample mesh for 2D channel flow past a block.}
\end{figure}

The external force is taken to be $f=0$. We use  $Re$=$\frac{UL}{\nu}$=100 and 150, which are calculated using the length scale $L=0.1$ as the width of the block and $U=1$, and then $\nu$ is chosen appropriately to produce these $Re$.  $(P_2, P_1^{disc})$ Scott-Vogelius elements are used on a barycenter refined Delaunay mesh that provides 120K velocity and 89K pressure dof.  Figure \ref{cylpic} shows a mesh of the domain significantly coarser than what we use in the tests below.  Figure \ref{cylref} shows a $Re$=100 steady solution as speed contours.   We note that both $Re$=100 and 150 are known to produce periodic-in-time solutions \cite{HRV24}.

\begin{figure}[H]
  \centering
       \includegraphics[scale=0.34]{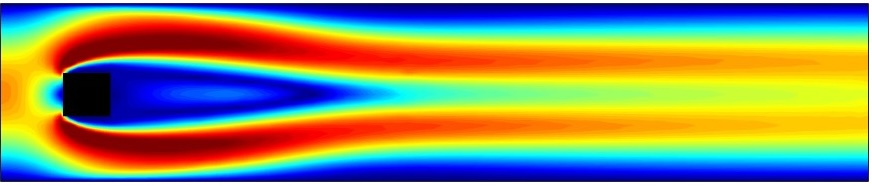} 
  \caption{\label{cylref} Shown above is a reference solution for 2D channel flow past a block with $Re$=100 shown as speed contours.}
\end{figure}

\subsection{Tests of CDA-IPY with exact partial solution data}

We now test CDA-IPY on the 2D driven cavity problem using exact partial solution data.  We compute using varying $H$, two Reynolds numbers $Re$=5000 and 1000, two meshes $h=\frac{1}{256}$ and $\frac{1}{392}$, and for comparison, we also compute CDA-Picard solutions and Picard solutions.  

Results for the 2D cavity using $Re$=5,000 are shown in Figure \ref{conv12}. We observe that for CDA-IPY $H=1/8$ does not help convergence (its convergence plot lies on top of Picard's), but $H=1/16$ does accelerate convergence significantly, and $H=1/32$ accelerates it even more.  These tests used $\mu=1$ and $\gamma=1$, but we also computed with larger $\mu$ (10, 100, 1000, and higher) and got the same results.   Computing with $\gamma=10$ gave the same results, but larger $\gamma$ values caused problems with the iterative solver convergence for the Schur complement linear system.  We also ran the same tests with CDA-Picard and varying $H$, and in Figure \ref{conv1} we observe that these results are nearly identical to those of CDA-IPY. Hence, using the more efficient CDA-IPY does not suffer in its convergence rate compared to CDA-Picard.  Lastly, we note that the convergence results are similar but slightly better on the finer mesh.

\begin{figure}[ht]
\center
$Re$=5000, $h=\frac{1}{256}$  \hspace{1.5in} $Re$=5000,   $h=\frac{1}{392}$ \\
\includegraphics[width = .48\textwidth, height=.28\textwidth,viewport=20 0 710 350, clip]{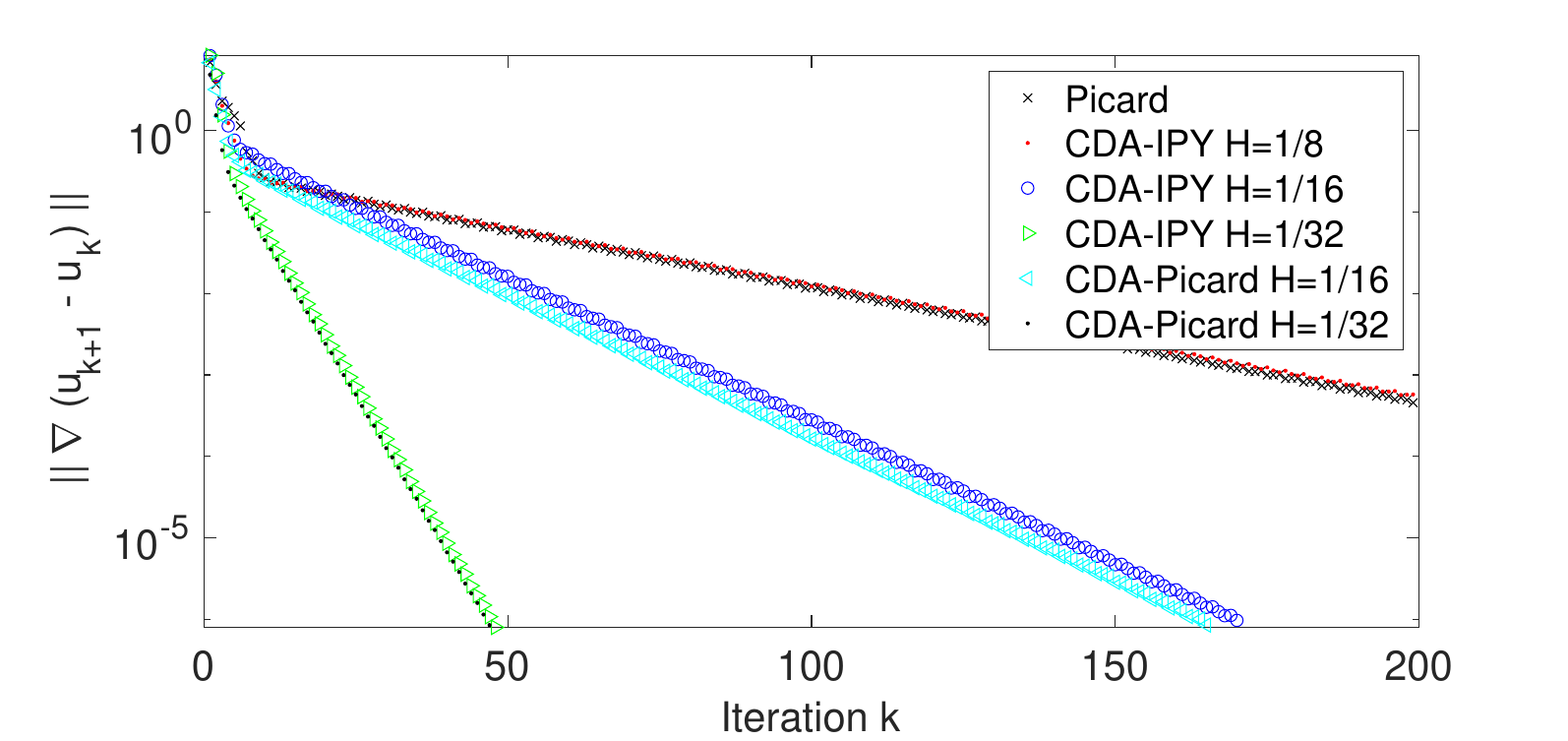}  
\includegraphics[width = .48\textwidth, height=.28\textwidth,viewport=20 0 710 350, clip]{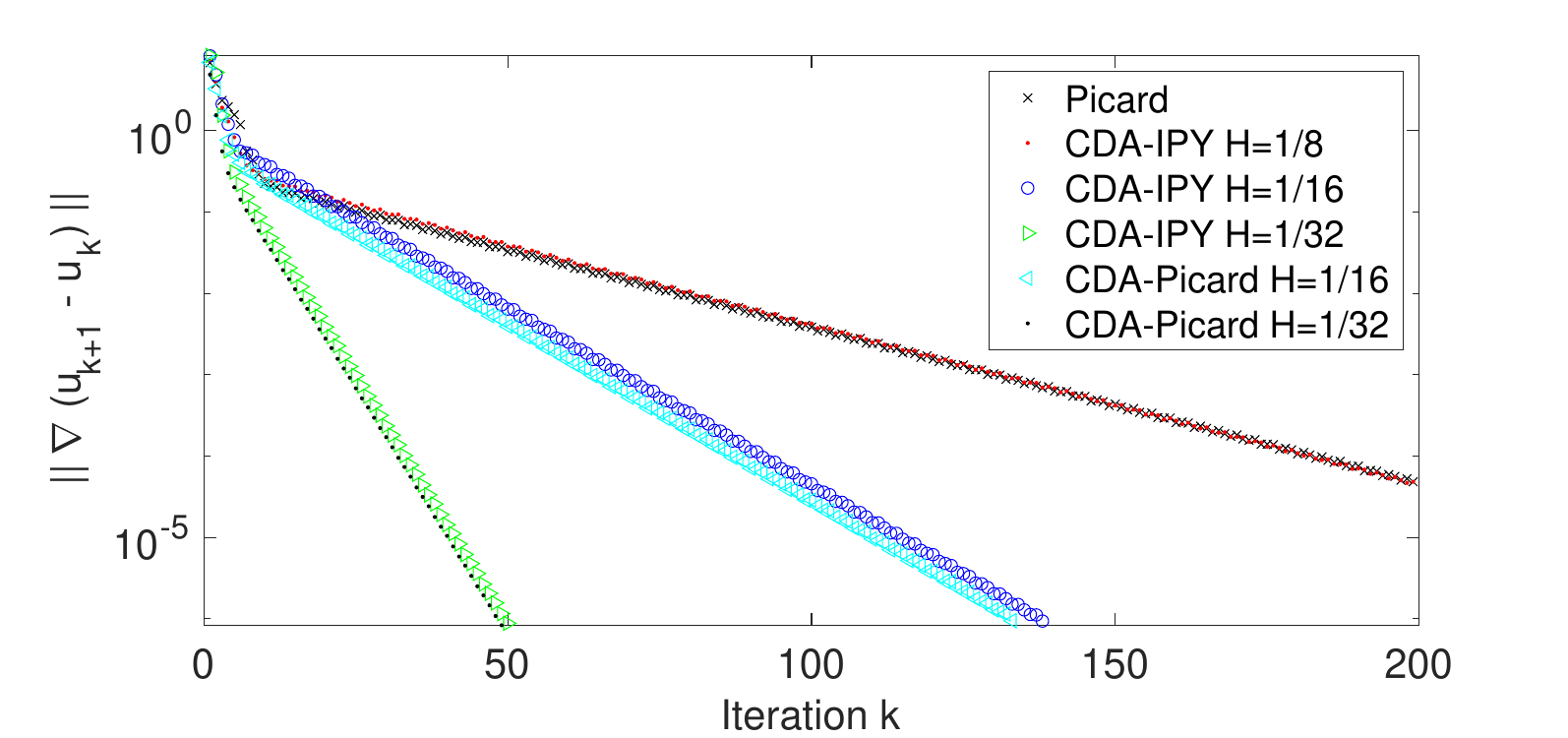}  
\caption{\label{conv12} The plots above show convergence of CDA-IPY and CDA-Picard for 2D driven cavity with $Re$=5,000 for $h=\frac{1}{256}$ (left) and $h=\frac{1}{392}$(right), varying $H$, and exact partial solution data.}
\end{figure}

Results for $Re$=10,000 2D cavity are shown in Figure \ref{conv22} for $\mu=1$ and $\gamma=1$.  We observe that Picard fails, CDA-IPY $H=1/32$ appears to provide convergence eventually (perhaps in about 1000 or so iterations using extrapolation), and CDA-IPY with $H=1/48$ and $1/64$ converge reasonably fast.  Results for $1/64$ are slightly better on the finer mesh, while the $1/48$ results are slightly worse.  Similar to the case of $Re$=5000 above, results for CDA-Picard are nearly identical to CDA-IPY and are therefore omitted.  Results for larger $\mu$ and/or $\gamma=10$ (instead of $\mu=\gamma=1$) give the same convergence plots and are also omitted.

\begin{figure}[ht]
\center
$Re$=10000, $h=\frac{1}{256}$  \hspace{1.5in} $Re$=10000,   $h=\frac{1}{392}$ \\
\includegraphics[width = .48\textwidth, height=.28\textwidth,viewport=20 0 710 350, clip]{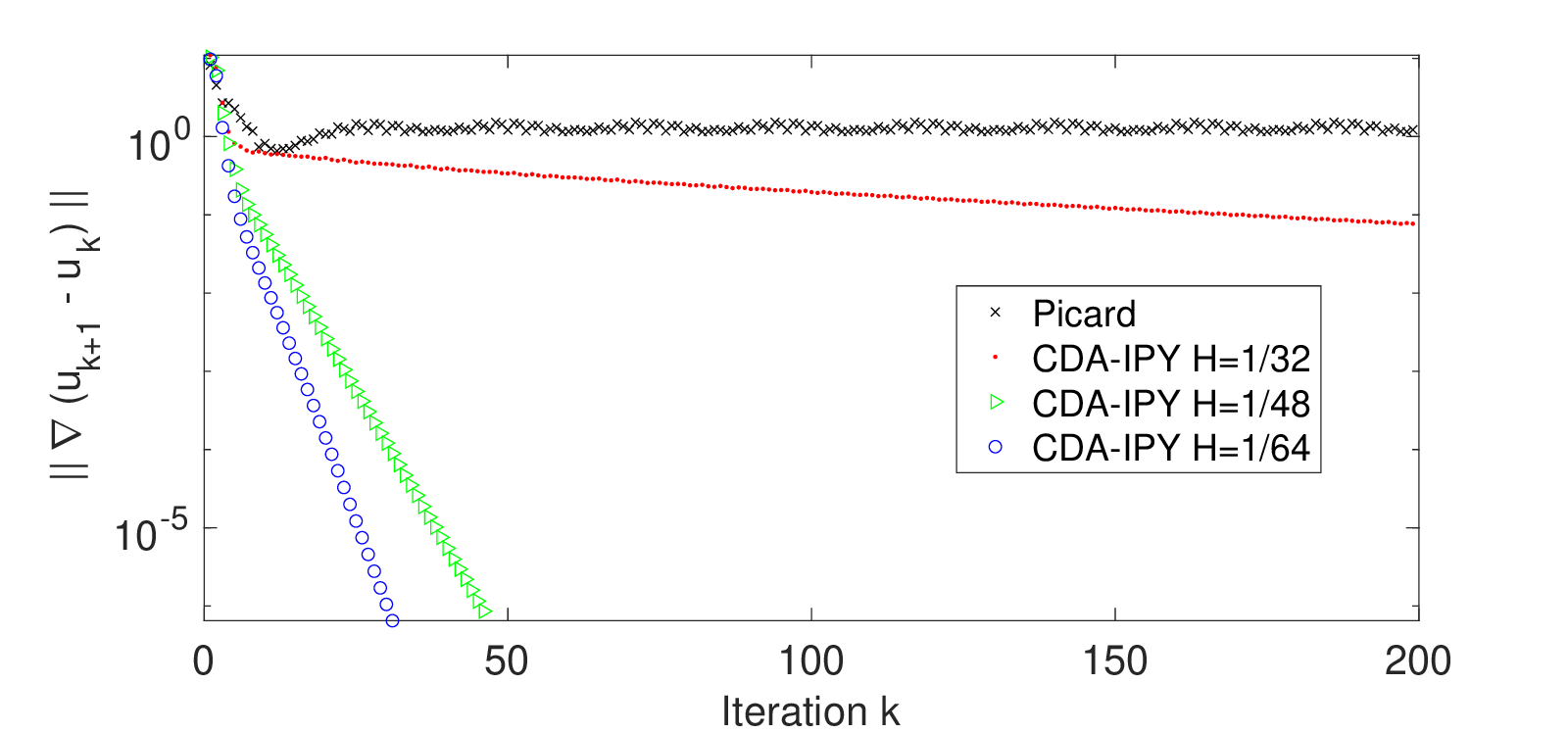}  
\includegraphics[width = .48\textwidth, height=.28\textwidth,viewport=20 0 710 350, clip]{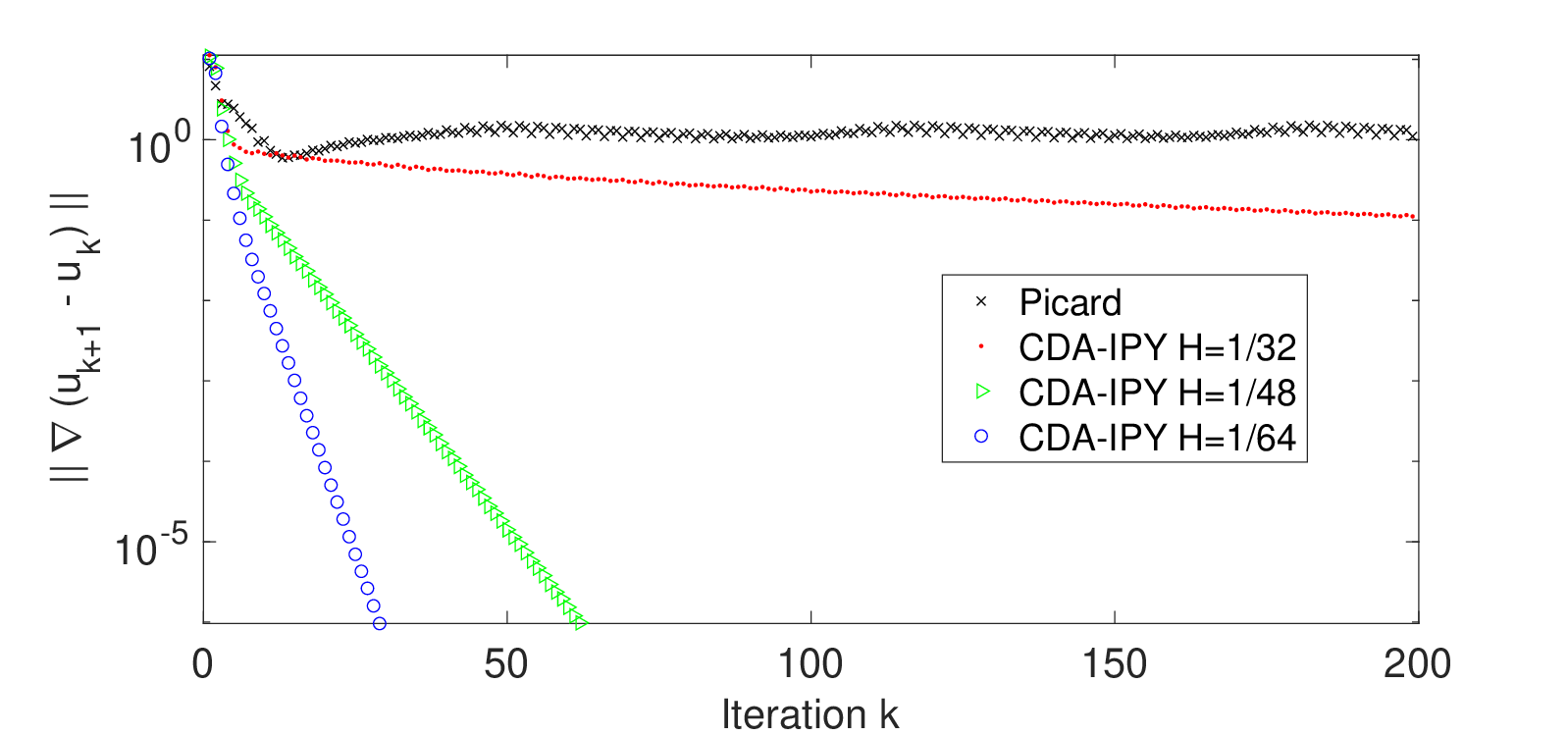}  
\caption{\label{conv22} The plots above show convergence of CDA-IPY for 2D driven cavity with $Re$=10,000 for $h=\frac{1}{256}$ (left) and $h=\frac{1}{392}$(right), varying $H$, and exact partial solution data.}
\end{figure}

Convergence results for the 3D cavity are shown in Figure \ref{conv3}, and CDA-IPY is observed to perform very well.  Picard converges (eventually) for $Re$=200 but not for 300 or above, so it is only included in the $Re$=200 plot.  For $Re$=1000, the maximum $H$ that provides convergence is $1/7$, and for $Re$=1500, it is $1/9$.  Taking $H\le 1/12$ is highly effective for these tests.  CDA-Picard was run with these same parameters, and we obtained very similar convergence results to CDA-IPY (plots omitted).

\begin{figure}[ht]
\center
$Re$=200 \hspace{2in} $Re$=1000  \\
\includegraphics[width = .48\textwidth, height=.28\textwidth,viewport=20 0 710 350, clip]{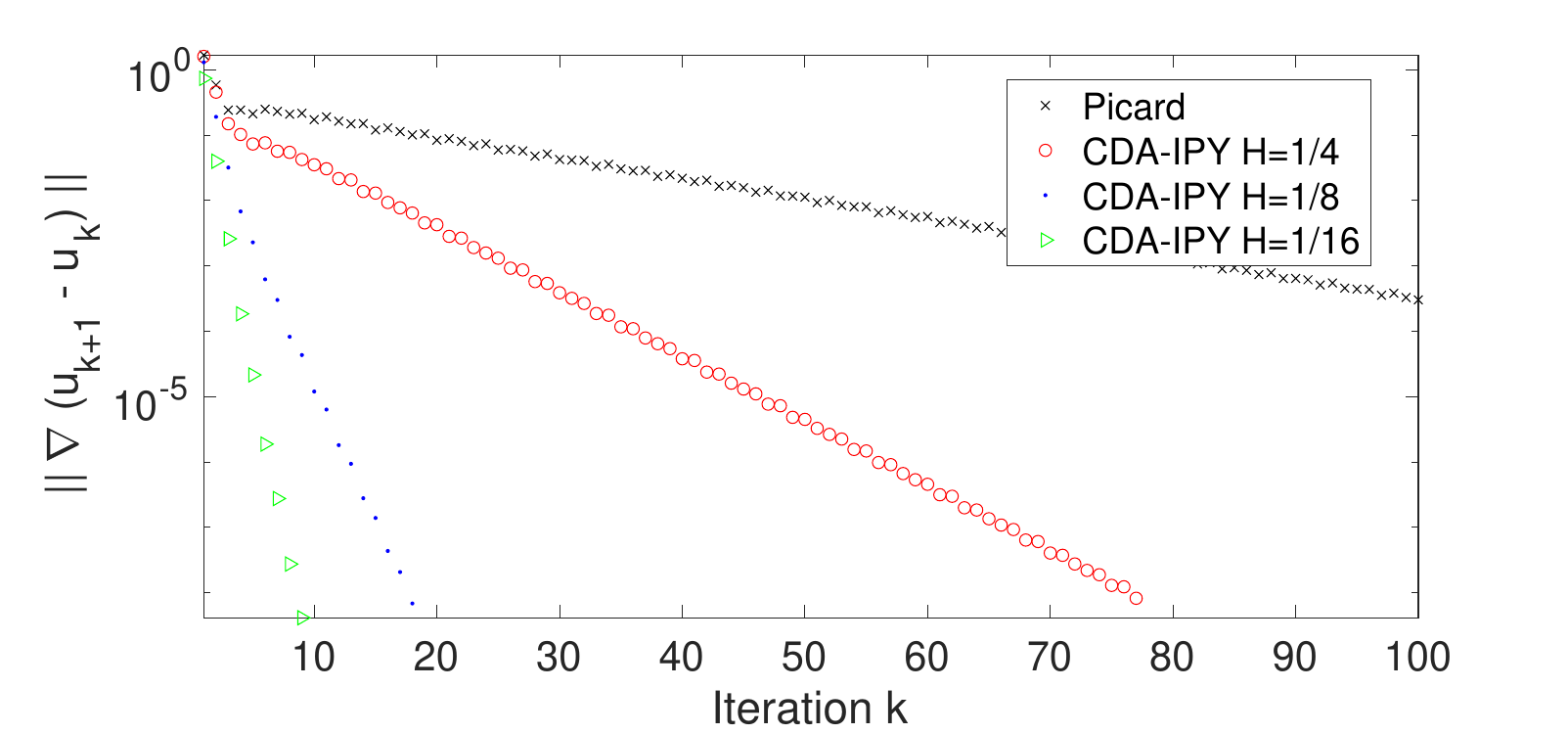}  
\includegraphics[width = .48\textwidth, height=.28\textwidth,viewport=20 0 710 350, clip]{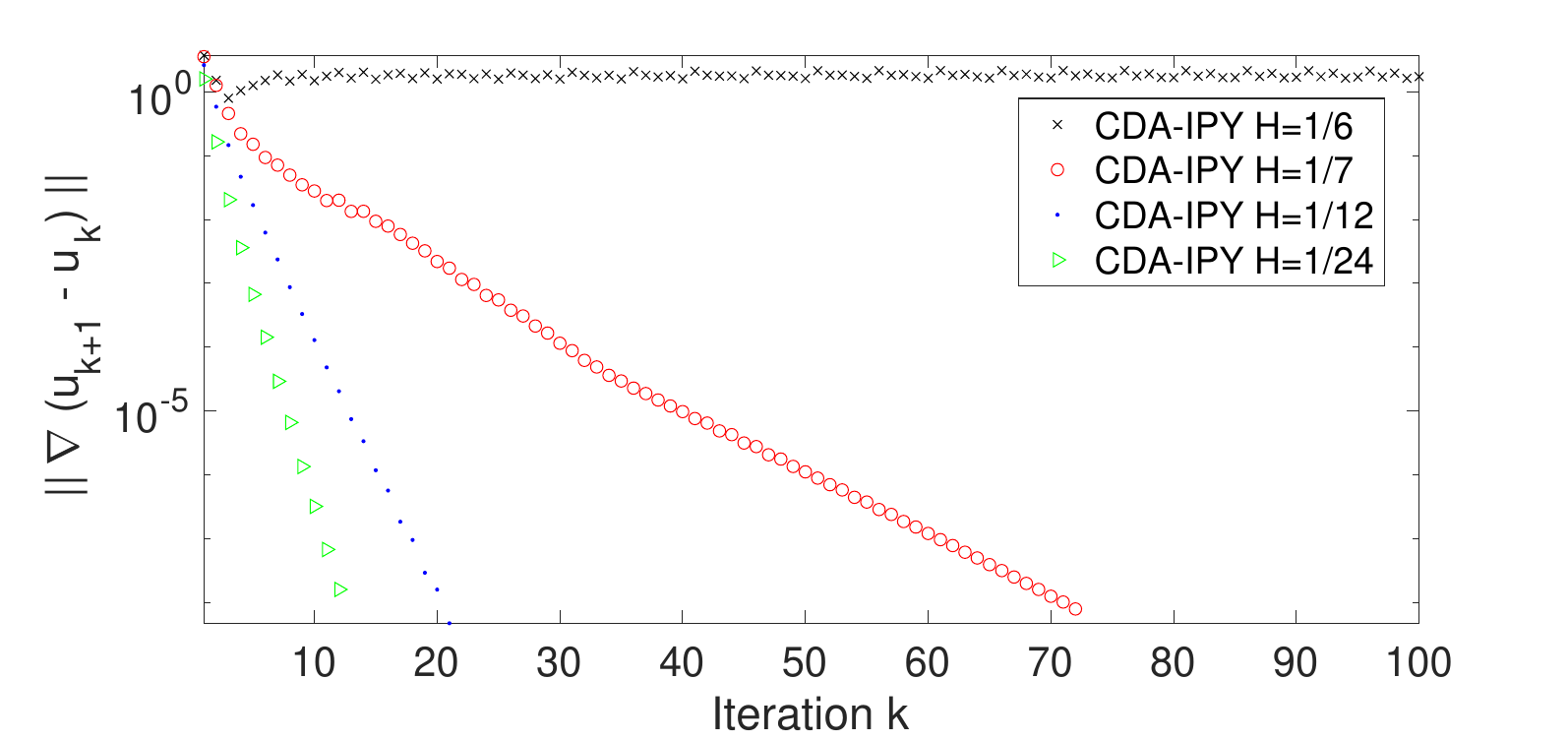}  \\
$Re$=1500\\
\includegraphics[width = .48\textwidth, height=.28\textwidth,viewport=20 0 710 350, clip]{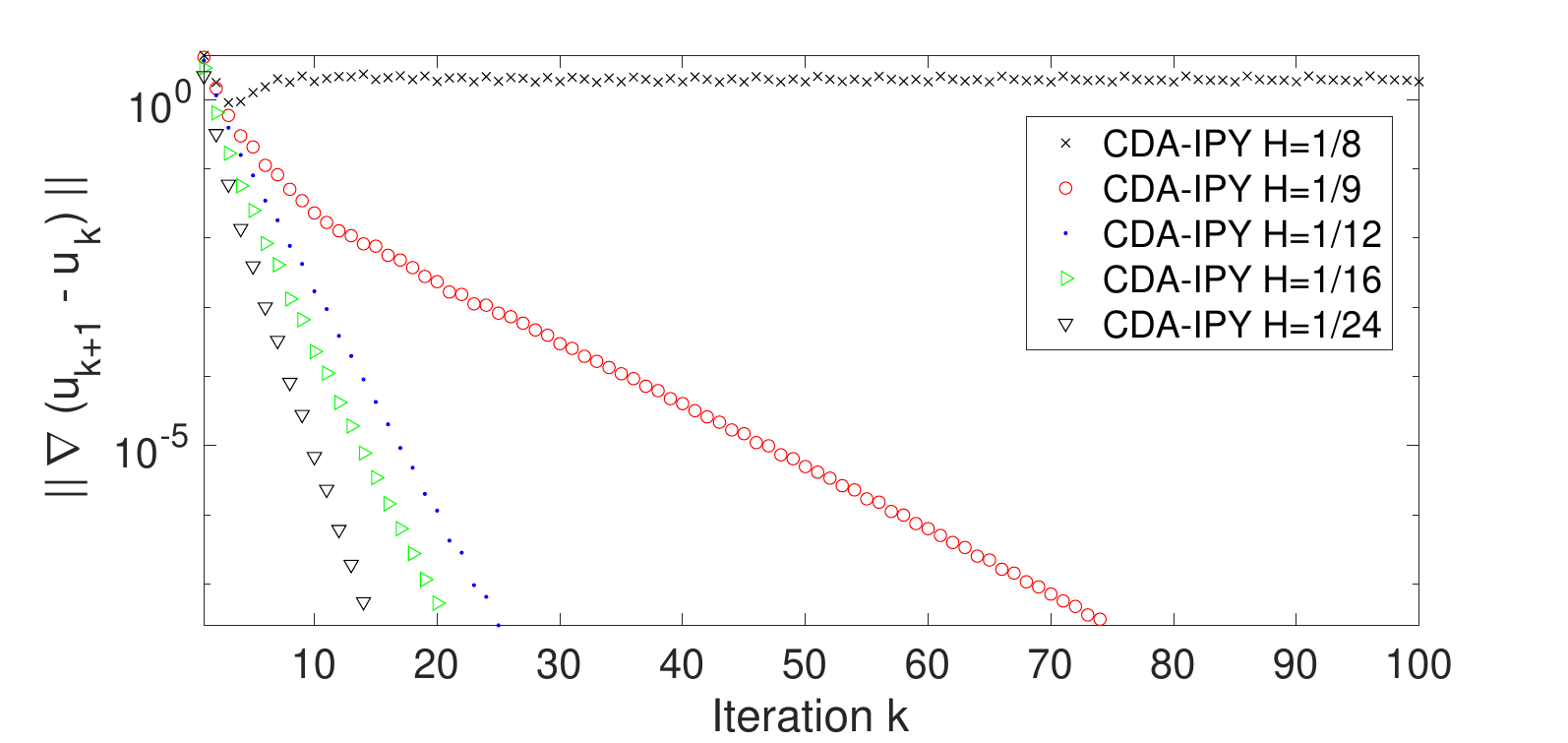}  
\caption{\label{conv3} The plots above show convergence of CDA-IPY for varying $Re$ for the 3D driven cavity problem.}
\end{figure}

Convergence plots for channel flow past a block are shown in Figure \ref{block1}.  For $Re$=100 and 150, they show the gain in convergence speed for CDA-IPY for decreasing $H$.  For $Re$=150, Picard and IPY fail to converge, and it is CDA that allows for convergence.

\begin{figure}[ht!]
\center
$Re$=100  \hspace{2.5in} $Re$=150 \\
\includegraphics[width = .48\textwidth, height=.28\textwidth,viewport=20 0 710 350, clip]{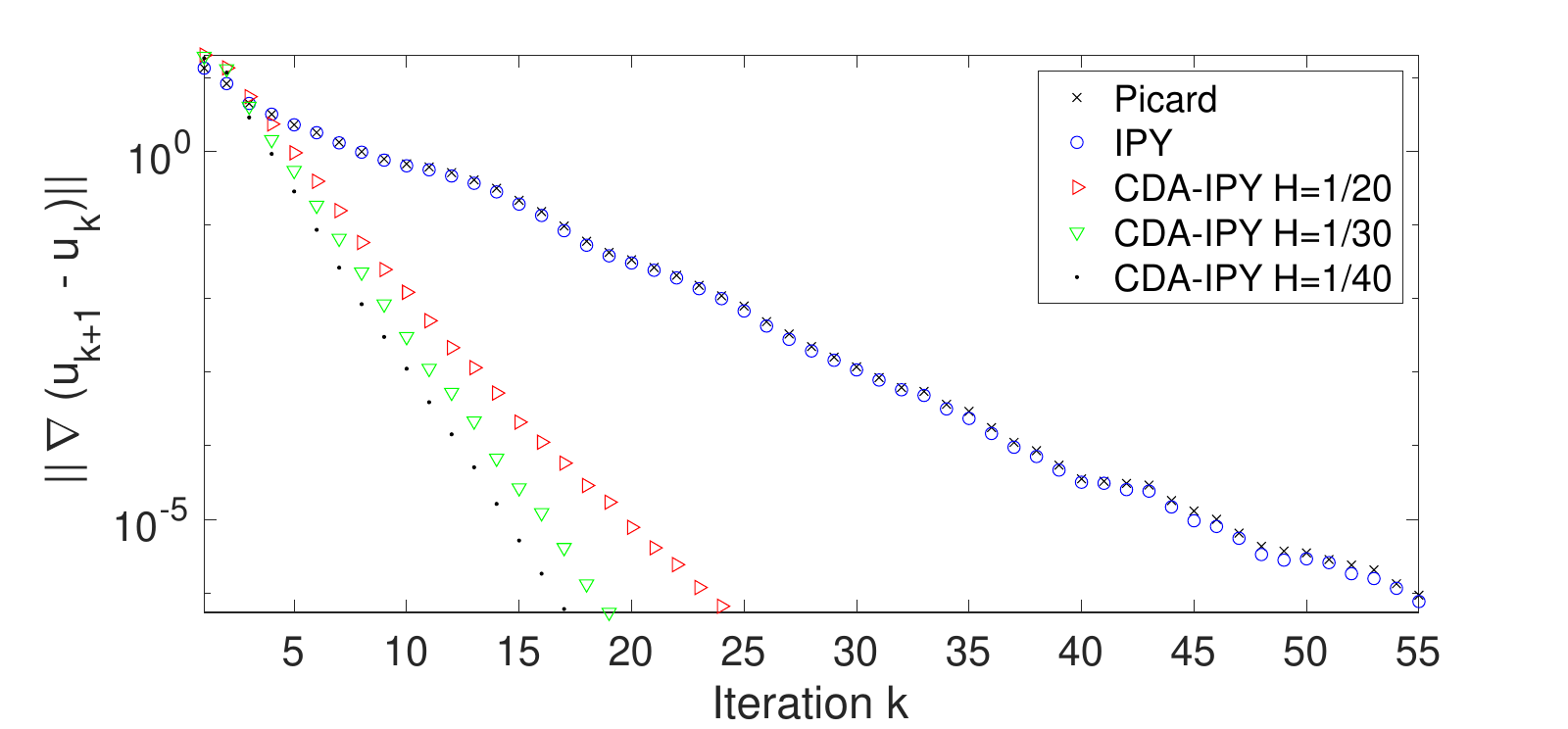}  
\includegraphics[width = .48\textwidth, height=.28\textwidth,viewport=20 0 710 350, clip]{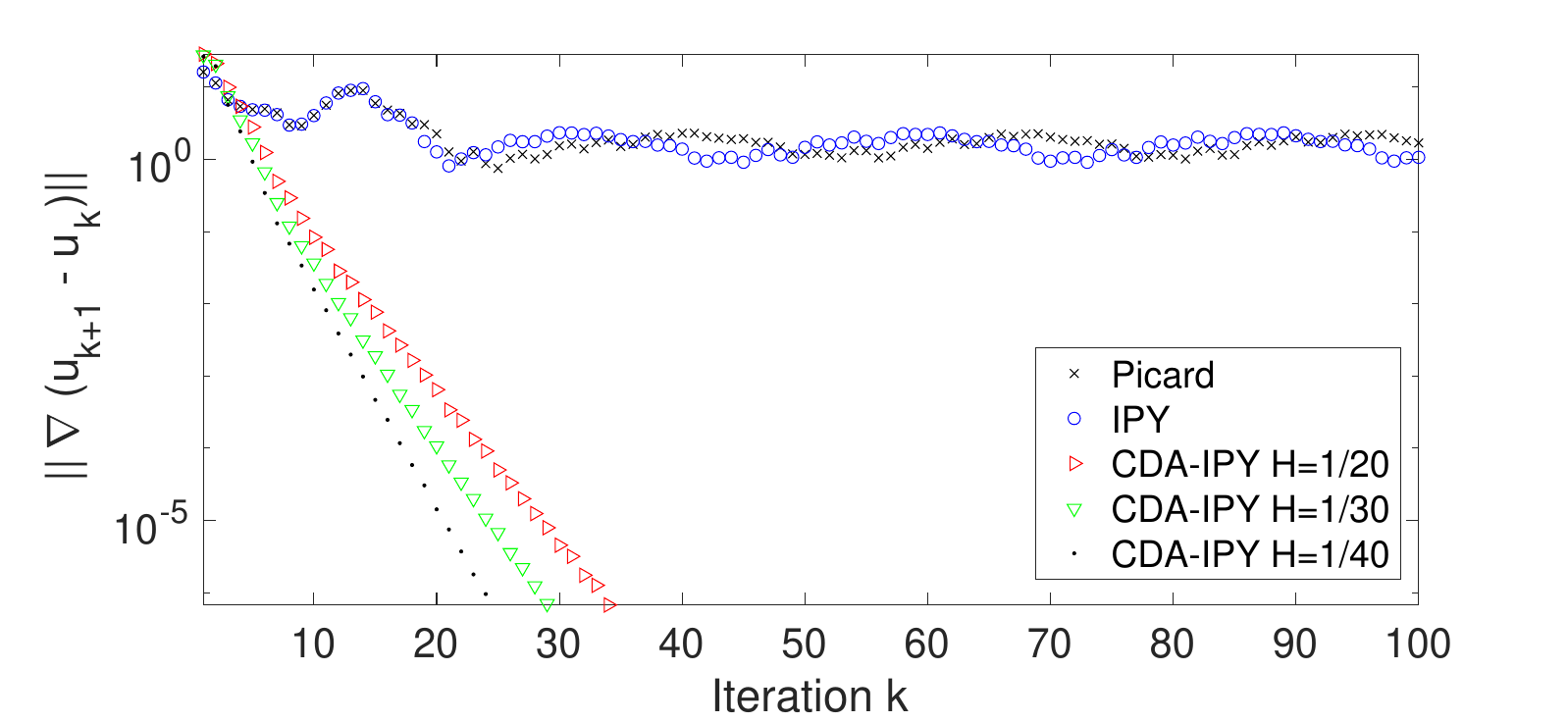}  
\caption{\label{block1} The plots above show convergence of CDA-IPY for 2D channel flow past a block with $Re$=100 (left) and 150 (right) for varying $H$ and exact partial solution data.}
\end{figure}

\subsection{Tests of CDA-IPY with noisy partial solution data}

We next illustrate the theory for CDA-IPY when the data contains noise.  Here, we expect the residual to converge to zero and the error to converge at the same rate until it bottoms out at the size of the noise.  We consider two experiments that generate noise in different ways.

For our first test, we consider CDA-IPY for the 2D channel flow past a block with $Re$=150 and $H$=1/30.  To generate the noise, we choose a signal to noise ratio (SNR). Then, we add component-wise uniform random numbers generated in [-1,1] scaled by SNR and then by the maximum absolute velocity.  This is then used to obtain the partial solution data for CDA-IPY.  Convergence results are shown in Figure \ref{noise12} for varying SNR.  In the figure on the left, we show both residuals and errors, and as expected, the residuals convergence linearly to zero, while the errors bottom out.  These errors may be considered unacceptably large, and thus it makes sense to try to improve them.  Hence, we take the first iterate from CDA-IPY that falls below $10^{-2}$ (to make sure it has bottomed out) and use it to seed the usual Newton iteration.  Results from this iteration are shown in the figure on the right, and we observe that while (unstabilized) Newton fails (blows up) and CDA-IPY + Newton fails with the largest SNR of 0.05, the CDA-IPY + Newton method is effective for SNR$\le 0.01$ as convergence is reached quickly after switching to Newton.  Hence, if the noise is not `too large', then CDA-IPY+Newton is an effective method.

\begin{figure}[ht]
\center
CDA-IPY  \hspace{2in} CDA-IPY + Newton  \\
\includegraphics[width = .48\textwidth, height=.28\textwidth,viewport=0 0 710 350, clip]{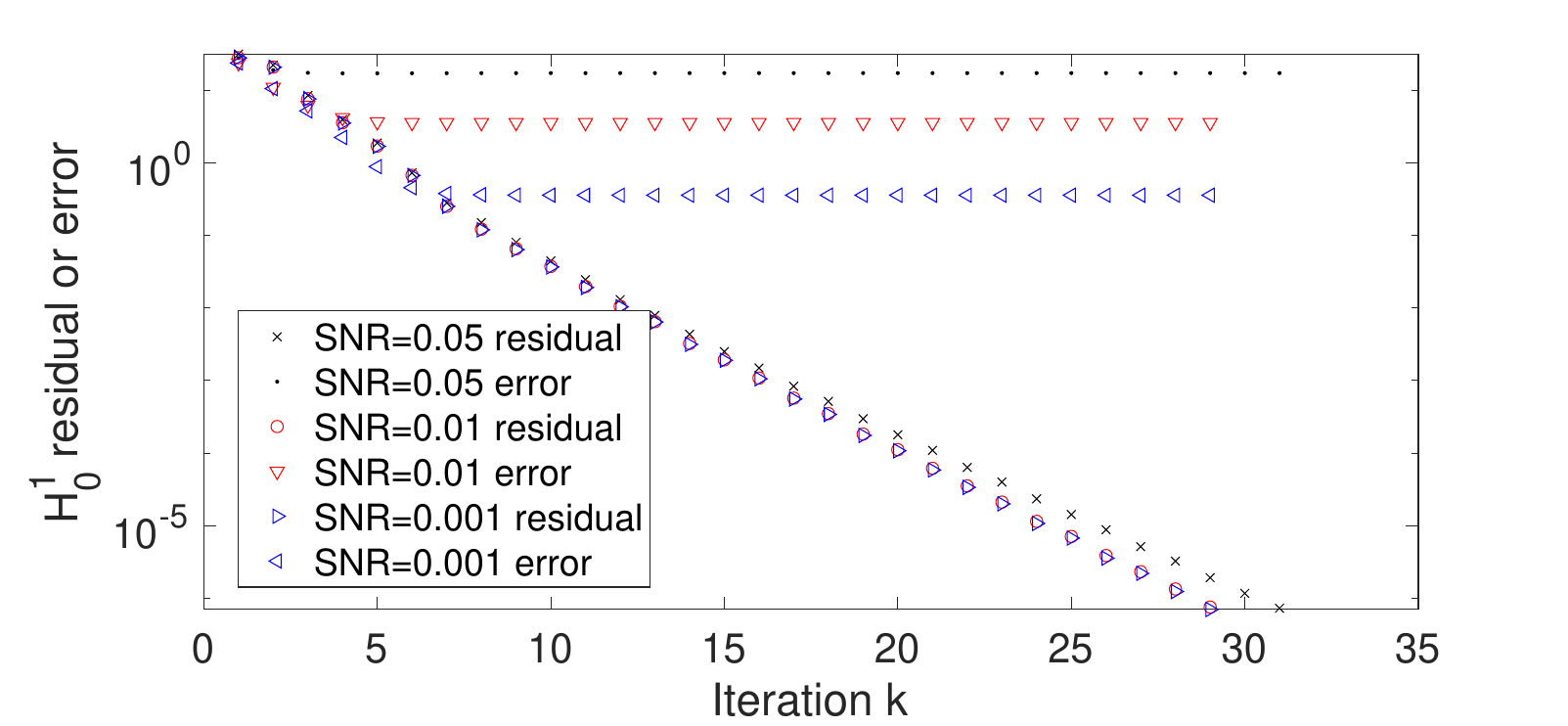}  
\includegraphics[width = .48\textwidth, height=.28\textwidth,viewport=0 0 710 350, clip]{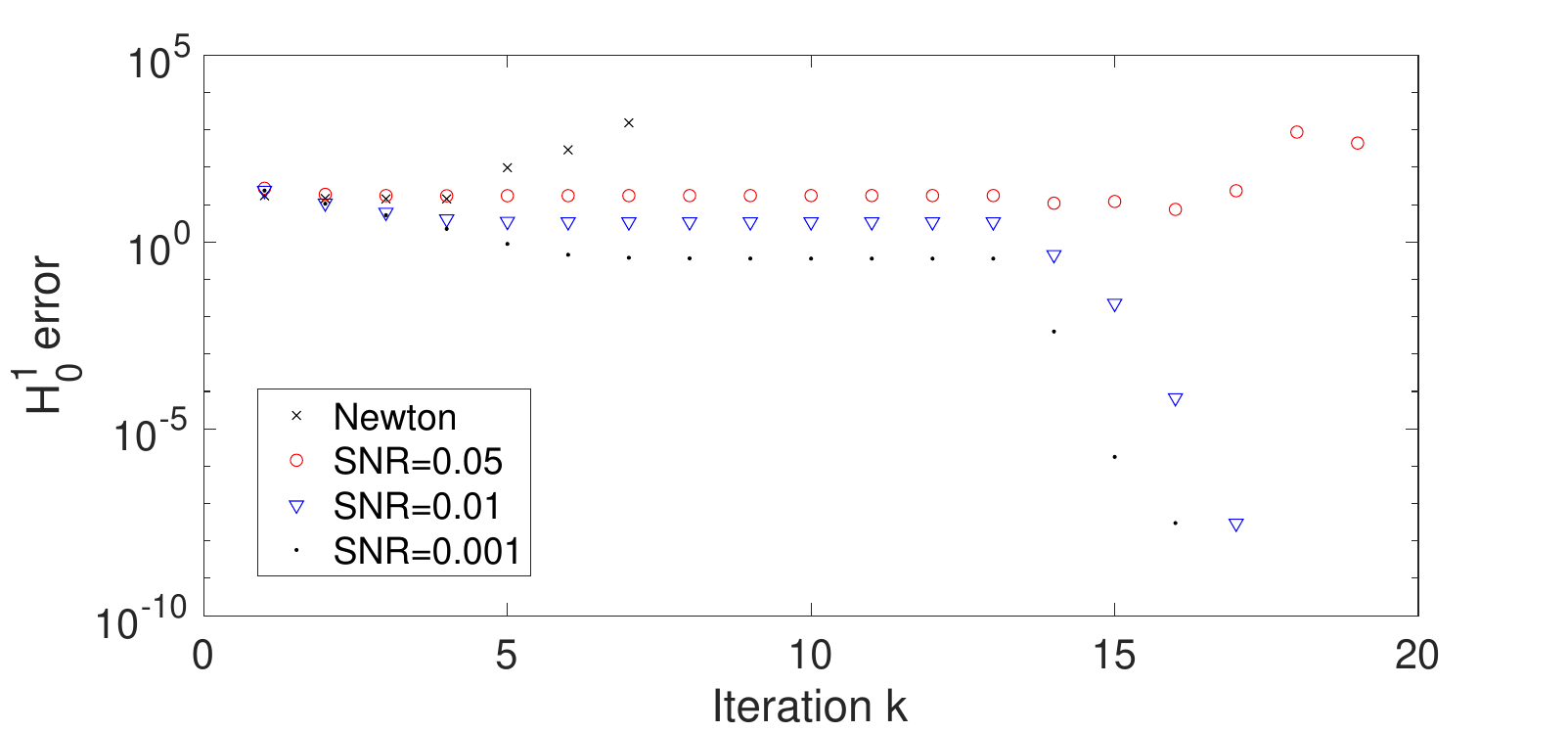}  
\caption{\label{noise12} The plots above show convergence of CDA-IPY for 2D channel flow past a block with $Re$=150 (left) and CDA-IPY+Newton [switch to Newton with no CDA after velocity residual is below $10^{-2}$] for the same problems (right) with noisy partial solution data.}
\end{figure}
%
%
%
%
%
%
%

For our second test with partial solution data that is not exact, we consider the 2D driven cavity with $Re$=5000 and $h=$1/196 with $H$=1/32.  Here, we mimic the setting of passing minimal solution data from one user to another by first taking the $32^2\times 1$ partial solution vector and generating a $32\times 32$ matrix.  An SVD of this matrix is then constructed, and a reduced SVD is then passed to another user who recovers the partial solution approximately by reconstruction from the reduced SVD (the approximation error is thus the noise). Convergence results are shown in Figure \ref{noise2}.  On the left of the figure, we plot the errors and residuals for varying ranks $r$ of reduced SVDs (r=2,4,8,16).  This is done for both the $x$ and $y$ components of the solution.

Just as in the first example with noisy partial solution data, and in agreement with the theory, the residuals converge linearly to zero but the errors bottom out (and as expected, the higher the $r$, the lower the error bottom).  On the right in Figure \ref{noise2} is CDA-IPY + Newton, where similar to the previous example, we run CDA-IPY in the same manner. Once the residual falls below $10^{-2}$ in the $H^1_0$ norm, we switch to Newton (note that Newton alone will not converge at this $Re$ with the same initial guess of 0).  For each case of $r$, once the switch to Newton is made, convergence of the error is achieved in just 3 or 4 iterations.  In other words, in this case of inaccurate partial solution data, the data was still sufficiently accurate; after CDA-IPY reached its best accuracy using the data, it was sufficiently close that Newton could quickly converge to the true solution.

Table \ref{noise3} shows the dramatic reduction in data that is possible with the CDA-IPY + Newton method.  Even without compression, the data that needed to be passed from the $2*32^2$ partial solution data was sufficient to reconstruct the solution with about half a percent of the amount of original data (2048 floating point numbers instead of 394,242).  With SVD compression, the amount of data that needs to be passed can be significantly reduced further, and the true solution is still recoverable.

\begin{figure}[ht]
\center
CDA-IPY  \hspace{2in} CDA-IPY + Newton  \\
\includegraphics[width = .48\textwidth, height=.28\textwidth,viewport=0 0 710 350, clip]{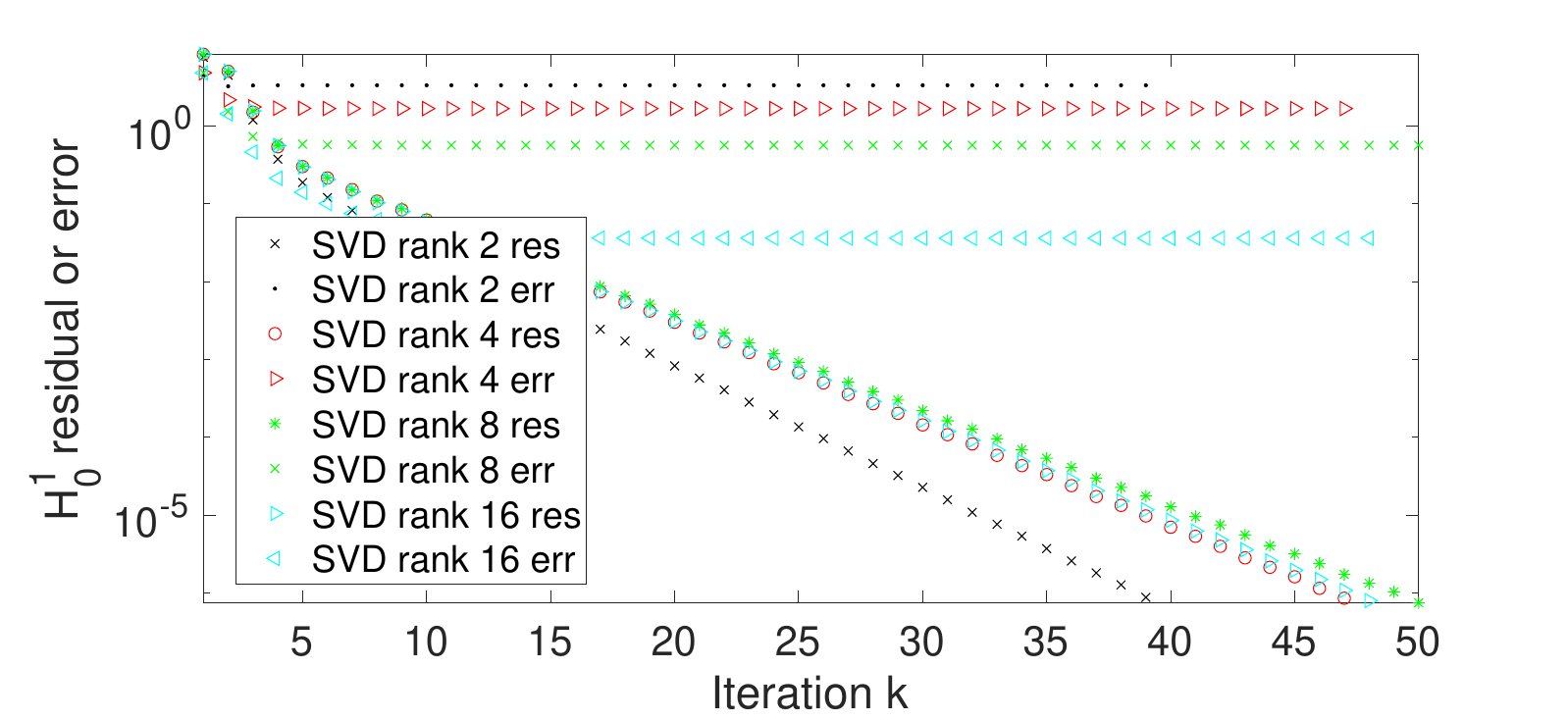}  
\includegraphics[width = .48\textwidth, height=.28\textwidth,viewport=0 0 710 350, clip]{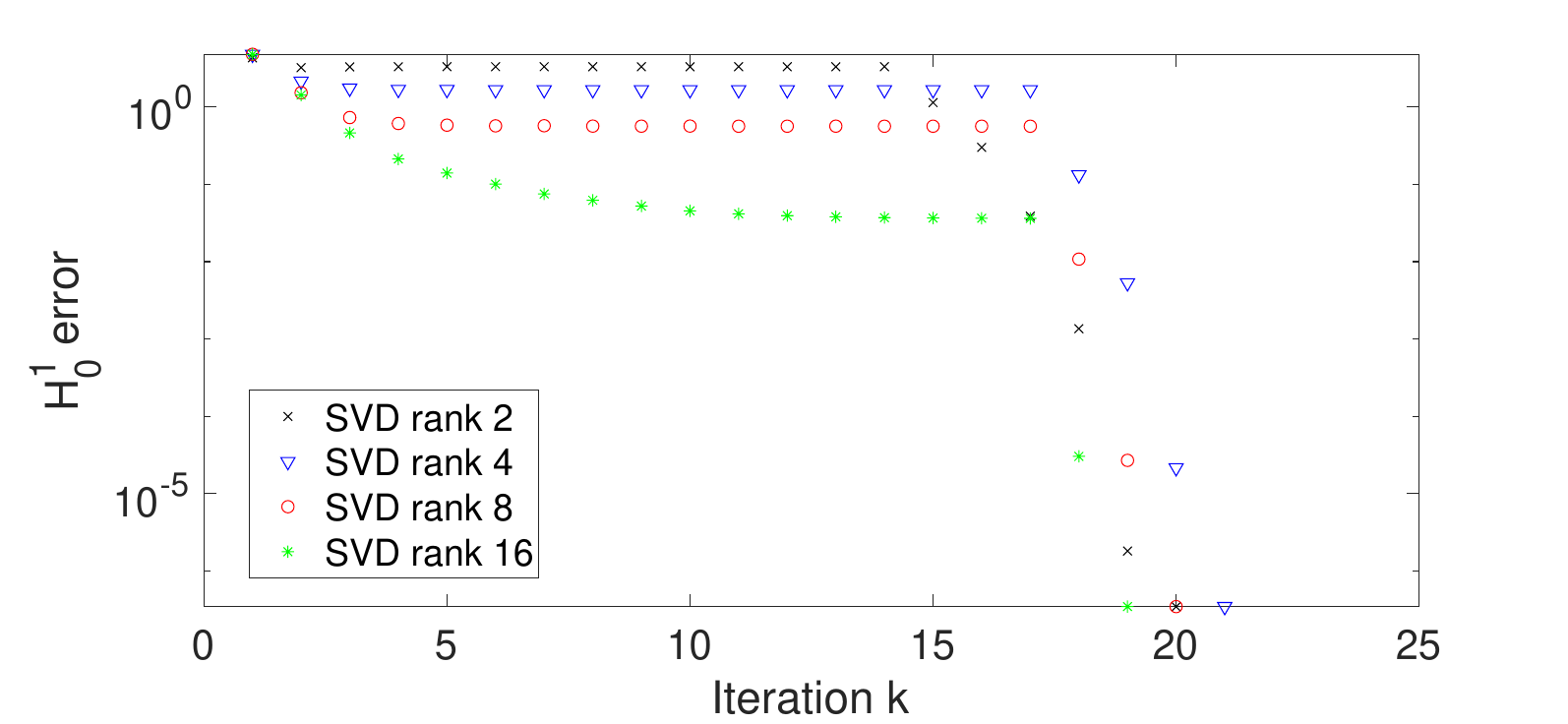}  
\caption{\label{noise2} The plots above show convergence of CDA-IPY for $Re$=5000 2D driven cavity flow with $H$=1/32 and varying SVD rank for partial solution data recovery.  Shown on the left are the errors and residuals for varying ranks, and on the right are errors for a hybrid strategy where CDA-IPY is used while the residual is less than $10^{-2}$, and then (usual) Newton is used afterward.
}
\end{figure}

\begin{table}[H]
\begin{tabular}{lllllll}
 & SVD r=2      & SVD r=4    & SVD r=8    & SVD r=16  & no SVD reduction \\
entries passed    & 260 entries          & 520   entries     & 1040 entries      & 2080   entries  & 2048   \\
\% of original data                      & 0.0660 \% & 0.1319\% & 0.2638\% & 0.5276\% &  0.51948\%
\end{tabular}
\caption{\label{noise3}Shown above are the total entries/dof (floating point numbers) that need to be passed for CDA-IPY using SVD compression.  The total number of dof of the computed solution is 394,242. }
\end{table}

\section{Conclusions}

We have given a detailed analytical and numerical study of the CDA-IPY method for incorporating data into an efficient and effective nonlinear solver.  CDA-Picard was proposed recently for this purpose \cite{LHRV23,GLNR24} and was found to be effective.  We have shown that CDA-IPY is equally effective as CDA-Picard; however, CDA-IPY is much more efficient - the linear systems at each iteration are much simpler, and there is no loss in convergence from using CDA-IPY instead of CDA-Picard. Hence, overall efficiency is significantly improved.  For future work, we plan to consider even more efficient splitting-type methods that are generally much less effective than Picard (e.g. Uzawa and Arrow-Hurwicz) to see if CDA makes them more robust.

\bibliographystyle{siam}
\bibliography{main}

\begin{thebibliography}{10}

\bibitem{arnold:qin:scott:vogelius:2D}
{\sc D.~Arnold and J.~Qin}, {\em Quadratic velocity/linear pressure {S}tokes elements}, in Advances in Computer Methods for Partial Differential Equations VII, R.~Vichnevetsky, D.~Knight, and G.~Richter, eds., IMACS, 1992, pp.~28--34.

\bibitem{AOT14}
{\sc A.~Azouani, E.~Olson, and E.~S. Titi}, {\em Continuous data assimilation using general interpolant observables}, Journal of Nonlinear Science, 24 (2014), pp.~277--304.

\bibitem{benzi}
{\sc M.~Benzi and M.~Olshanskii}, {\em An augmented {L}agrangian-based approach to the {O}seen problem}, SIAM J. Sci. Comput., 28 (2006), pp.~2095--2113.

\bibitem{BGHRR25}
{\sc C.~Bernardi, V.~Girault, F.~Hecht, P.-A. Raviart, and B.~Riviere}, {\em Mathematics and Finite Element Discretizations of Incompressible {N}avier-{S}tokes Flows}, Society for Industrial and Applied Mathematics, Philadelphia, PA, 2024.

\bibitem{BB12}
{\sc S.~B{\"o}rm and S.~Le~Borne}, {\em {$\mathcal H$}-{LU} factorization in preconditioners for augmented {L}agrangian and grad-div stabilized saddle point systems}, Internat. J. Numer. Methods Fluids, 68 (2012), pp.~83--98.

\bibitem{BS08}
{\sc S.~Brenner and L.~R. Scott}, {\em The Mathematical Theory of Finite Element Methods}, vol.~15 of Texts in Applied Mathematics, Springer Science+Business Media, LLC, 2008.

\bibitem{CHOR17}
{\sc T.~Charnyi, T.~Heister, M.~Olshanskii, and L.~Rebholz}, {\em On conservation laws of {N}avier-{S}tokes {G}alerkin discretizations}, Journal of Computational Physics, 337 (2017), pp.~289--308.

\bibitem{elman:silvester:wathen}
{\sc H.~Elman, D.~Silvester, and A.~Wathen}, {\em Finite Elements and Fast Iterative Solvers with applications in incompressible fluid dynamics}, Numerical Mathematics and Scientific Computation, Oxford University Press, Oxford, 2014.

\bibitem{ECG05}
{\sc E.~Erturk, T.~C. Corke, and C.~G\"okc\"ol}, {\em Numerical solutions of 2d-steady incompressible driven cavity flow at high {R}eynolds numbers}, Int. J. Numer. Methods Fluids, 48 (2005), pp.~747--774.

\bibitem{FMSW21}
{\sc P.~Farrell, L.~Mitchell, L.~Scott, and F.~Wechsung}, {\em A {R}eynolds-robust preconditioner for the {S}cott-{V}ogelius discretization of the stationary incompressible {N}avier-{S}tokes equations}, SMAI Journal of Computational Mathematics, 7 (2021), pp.~75--96.

\bibitem{GLRW12}
{\sc K.~Galvin, A.~Linke, L.~Rebholz, and N.~Wilson}, {\em Stabilizing poor mass conservation in incompressible flow problems with large irrotational forcing and application to thermal convection}, Computer Methods in Applied Mechanics and Engineering, 237 (2012), pp.~166--176.

\bibitem{GLNR24}
{\sc B.~Garcia-Archilla, X.~Li, J.~Novo, and L.~Rebholz}, {\em Enhancing nonlinear solvers for the {N}avier-{S}tokes equations with continuous (noisy) data assimilation}, Computer Methods in Applied Mechanics and Engineering, 424 (2024), pp.~1--15.

\bibitem{GN20}
{\sc B.~Garcia-Archilla and J.~Novo}, {\em Error analysis of fully discrete mixed finite element data assimilation schemes for the {N}avier-{S}tokes equations}, Advances in Computational Mathematics,  (2020), pp.~46--61.

\bibitem{GR86}
{\sc V.~Girault and P.-A.Raviart}, {\em Finite element methods for Navier-Stokes equations: Theory and Algorithms}, Springer-Verlag, 1986.

\bibitem{GR79}
{\sc V.~Girault and P.-A. Raviart}, {\em Finite element approximation of the Navier-Stokes equations}, vol.~749 of Lecture Notes in Mathematics, Springer-Verlag, Berlin, 1979.

\bibitem{GS98}
{\sc P.~Gresho and R.~Sani}, {\em Incompressible {F}low and the {F}inite {E}lement {M}ethod}, vol.~2, Wiley, 1998.

\bibitem{H25}
{\sc E.~Hawkins}, {\em Accelerating convergence of a natural convection solver by continuous data assimilation}, Submitted,  (2025).

\bibitem{HRV24}
{\sc E.~Hawkins, L.~Rebholz, and D.~Vargun}, {\em Removing splitting/modeling error in projection/penalty methods for {N}avier-{S}tokes simulations with continuous data assimilation}, Communications in Mathematical Research, 40 (2024), pp.~1--29.

\bibitem{HO25}
{\sc Y.~He and M.~Olshanskii}, {\em A preconditioner for the grad-div stabilized equal-order finite elements discretizations of the {O}seen problem}, SIAM Journal on Scientific Computing, 47 (2025), pp.~A1486--A1506.

\bibitem{HR13}
{\sc T.~Heister and G.~Rapin}, {\em Efficient augmented {L}agrangian-type preconditioning for the {O}seen problem using grad-div stabilization}, Int. J. Numer. Meth. Fluids, 71 (2013), pp.~118--134.

\bibitem{JJLR13}
{\sc E.~Jenkins, V.~John, A.~Linke, and L.~Rebholz}, {\em On the parameter choice in grad-div stabilization for the stokes equations}, Advances in Computational Mathematics, 40 (2014), pp.~491--516.

\bibitem{JLMNR17}
{\sc V.~John, A.~Linke, C.~Merdon, M.~Neilan, and L.~G. Rebholz}, {\em On the divergence constraint in mixed finite element methods for incompressible flows}, SIAM Review, 59 (2017), pp.~492--544.

\bibitem{LRZ19}
{\sc A.~Larios, L.~Rebholz, and C.~Zerfas}, {\em Global in time stability and accuracy of {IMEX-FEM} data assimilation schemes for {Navier-Stokes} equations}, Computer Methods in Applied Mechanics and Engineering, 345 (2019), pp.~1077--1093.

\bibitem{Laytonbook}
{\sc W.~Layton}, {\em An {I}ntroduction to the {N}umerical {A}nalysis of {V}iscous {I}ncompressible {F}lows}, SIAM, Philadelphia, 2008.

\bibitem{LHRV23}
{\sc X.~Li, E.~Hawkins, L.~Rebholz, and D.~Vargun}, {\em Accelerating and enabling convergence of nonlinear solvers for {N}avier-{S}tokes equations by continuous data assimilation}, Computer Methods in Applied Mechanics and Engineering, 416 (2023), pp.~1--17.

\bibitem{LRX24}
{\sc J.~Liu, L.~Rebholz, and M.~Xiao}, {\em Efficient and effective algebraic splitting iterations for nonlinear saddle point problems}, Mathematical Methods in the Applied Sciences, 47 (2024), pp.~451--474.

\bibitem{O02}
{\sc M.~A. Olshanskii}, {\em {A low order Galerkin finite element method for the Navier-Stokes equations of steady incompressible flow: a stabilization issue and iterative methods}}, Comput. Meth. Appl. Mech. Eng., 191 (2002), pp.~5515--5536.

\bibitem{OR04}
{\sc M.~A. Olshanskii and A.~Reusken}, {\em {Grad-Div stabilization for the Stokes equations}}, Math. Comp., 73 (2004), pp.~1699--1718.

\bibitem{PR25}
{\sc S.~Pollock and L.~Rebholz}, {\em Anderson Acceleration for Numerical PDEs}, Society for Industrial and Applied Mathematics, Philadelphia, PA, 2025.

\bibitem{PRTX25}
{\sc S.~Pollock, L.~Rebholz, X.~Tu, and M.~Xiao}, {\em Analysis of the {P}icard-{N}ewton iteration for the {N}avier-{S}tokes equations: global stability and quadratic convergence}, Journal of Scientific Computing, 14:25 (2025), pp.~1--23.

\bibitem{PRX19}
{\sc S.~Pollock, L.~Rebholz, and M.~Xiao}, {\em Anderson-accelerated convergence of {P}icard iterations for incompressible {N}avier-{S}tokes equations}, SIAM Journal on Numerical Analysis, 57 (2019), pp.~615-- 637.

\bibitem{RVX19}
{\sc L.~Rebholz, A.~Viguerie, and M.~Xiao}, {\em Efficient nonlinear iteration schemes based on algebraic splitting for the incompressible {N}avier-{S}tokes equations}, Mathematics of Computation, 88 (2019), pp.~1533--1557.

\bibitem{RZ21}
{\sc L.~G. Rebholz and C.~Zerfas}, {\em Simple and efficient continuous data assimilation of evolution equations via algebraic nudging}, Numerical Methods for Partial Differential Equations, 37 (2021), pp.~2588--2612.

\bibitem{R97}
{\sc W.~Rodi}, {\em Comparison of {LES} and {RANS} calculations of the flow around bluff bodies}, Journal of Wind Engineering and Industrial Aerodynamics, 69-71 (1997), pp.~55--75.
\newblock Proceedings of the 3rd International Colloqium on Bluff Body Aerodynamics and Applications.

\bibitem{SDN99}
{\sc A.~Sohankar, L.~Davidson, and C.~Norberg}, {\em {Large Eddy Simulation of Flow Past a Square Cylinder: Comparison of Different Subgrid Scale Models }}, Journal of Fluids Engineering, 122 (1999), pp.~39--47.

\bibitem{T68}
{\sc R.~Temam}, {\em Une m\'ethode d'approximation de la solution des \'equations de {Navier-Stokes}}, Bulletin de la Soci\'et\'e Math\'ematique de France, 96 (1968), pp.~115--152.

\bibitem{temam}
\leavevmode\vrule height 2pt depth -1.6pt width 23pt, {\em {N}avier-{S}tokes equations}, Elsevier, North-Holland, 1991.

\bibitem{TGO15}
{\sc F.~Trias, A.~Gorobets, and A.~Oliva}, {\em Turbulent flow around a square cylinder at {R}eynolds number 22,000: A {DNS} study}, Computers \& Fluids, 123 (2015), pp.~87--98.

\bibitem{WongBaker2002}
{\sc K.~Wong and A.~Baker}, {\em A {3D} incompressible {Navier–Stokes} velocity–vorticity weak form finite element algorithm}, International Journal for Numerical Methods in Fluids, 38 (2002), pp.~99--123.

\bibitem{zhang:scott:vogelius:3D}
{\sc S.~Zhang}, {\em A new family of stable mixed finite elements for the {3D} {S}tokes equations}, Math. Comp., 74 (2005), pp.~543--554.

\end{thebibliography}

\end{document}